\documentclass[reqno,10pt]{amsart}

\usepackage{geometry}
\usepackage{enumerate,xspace}
\usepackage{amsmath}
\usepackage{amssymb}
\usepackage{amsthm}
\usepackage{stmaryrd}
\usepackage[colorlinks]{hyperref}
\usepackage[all,2cell]{xy}
\usepackage{mathrsfs,latexsym,graphicx,pifont}
\usepackage{framed}
\usepackage{xcolor}

\definecolor{shadecolor}{gray}{0.9}

\newenvironment{svgraybox}
  {\begin{shaded}}
  {\end{shaded}}

\xyoption{all}\xyoption{2cell}

\newtheorem{theorem}{Theorem}[section]
\newtheorem{lemma}[theorem]{Lemma}
\newtheorem{corollary}[theorem]{Corollary}
\newtheorem{proposition}[theorem]{Proposition}

\theoremstyle{definition}
\newtheorem{definition}[theorem]{Definition}
\theoremstyle{remark}
\newtheorem{example}[theorem]{Example}
\theoremstyle{remark}

\theoremstyle{remark}

\theoremstyle{remark}

\theoremstyle{remark}
\newtheorem{remark}[theorem]{Remark}
\theoremstyle{remark}
\newtheorem{exercise}[theorem]{Exercise}
\theoremstyle{remark}

\newcommand{\Set}{\ensuremath{\mathsf{Set}}}

\newcommand{\dd}{\ensuremath{\mathsf{d}}}
\newcommand{\cc}{\ensuremath{\mathsf{c}}}
\newcommand{\mm}{\ensuremath{\mathsf{m}}}

\newcommand{\Sh}{\mathsf{Sh}}

\newcommand{\id}{\mathtt{I}}

\newcommand{\zero}{\mathbf{0}}

\newcommand{\cA}{\mathcal{A}}
\newcommand{\G}{\mathcal{G}}
\newcommand{\cX}{\mathcal{X}}

\newcommand{\op}{\mathrm{op}}

\newcommand{\B}{{\ensuremath{\mathscr{B}}}\xspace}

\newcommand{\sG}{{\ensuremath{\mathscr{G}}}\xspace}

\newcommand{\eS}{{\ensuremath{\mathscr{S}}}\xspace}

\newcommand{\X}{{\ensuremath{\mathscr{X}}}\xspace}

\newcommand{\ra}{\longrightarrow}
\def\iso{\cong}

\def\adj{\dashv}

\begin{document}

\title{Locally involutive semigroups}

\author{Clemens Berger and Jonathon Funk}

\address{Universit\'e C\^ote d'Azur and CUNY Queensborough}
\email{clemens.berger@univ-cotedazur.fr and jfunk@qcc.cuny.edu}
\date{June 12, 2026}

\begin{abstract}We introduce locally involutive semigroups and embed them into the category of ordered groupoids.
This embedding restricts to a correspondence between quasi-involutive semigroups and ordered groupoids with mediator,
extending the classical ESN-correspondence between inverse semigroups and inductive groupoids.
An important subcategory of locally involutive semigroups is formed
by \emph{left involutive semigroups} because the classifying topos $\B(S)$ of an inverse semigroup $S$ is
equivalent to the category of left involutive semigroups \'etale over $S$ \cite{F7}.
We recover this equivalence from a general adjointness,
which we use to determine
when a left involutive semigroup \'etale over $S$ is actually an involutive semigroup.
Any left involutive semigroup \'etale over $S$ embeds
into an involutive $S$-algebra as we call it.
The underlying semigroup of this algebra is involutive.\end{abstract}

\maketitle

\newpage

\vspace*{\fill}

\tableofcontents

\vspace*{\fill}

\newpage

\section{Introduction}\label{sec:intro}
\vspace{2ex}

We introduce a new class of semigroups called locally involutive.
Locally involutive semigroups include the class known in the literature as regular $*$-semigroups,
but herein called involutive semigroups. Our first purpose is to establish
an ESN-correspondence between an intermediate class of quasi-involutive semigroups
and a class of ordered groupoids with mediator.
The mediator is a natural and basic idea that makes sense
even at the very general level of double categories.

The category of ordered actions of the so-called inductive groupoid
of an inverse semigroup $S$ is a topos,
called the classifying topos of $S$ and denoted $\B(S)\,$.
As the inductive groupoid with its Alexandrov topology is an \'etale groupoid,
$\B(S)$ is thus an instance of the
topos of \'etale actions of an \'etale groupoid
(or even generally with a localic groupoid).
The groupoid action approach can be simplified
producing a category of \'etale actions over the set of idempotents of $S$ \cite{FS}.
Another way to present $\B(S)$ uses presheaves on a certain left cancellative category $L(S)$ associated with $S\,$.
There is yet a fourth distinct way to present $\B(S)$
as a category of semigroups and homomorphisms over $S\,$.
This result is already published \cite{F7}, but here we give a new perspective
based on an adjointness between semigroups over $S$ and the classifying topos $\B(S)\,$.
This adjointness is shown to hold more generally when $S$ is a left involutive semigroup.

The adjointness at the center of this article is reminiscent of
the one for topological spaces over a space $X$ and its topos of sheaves $\Sh(X)\,$:
\[
\xymatrix{\text{Spaces}/X \ar@/_3ex/[rr]_{\Gamma} \ar@{}[rr]|{\bot}
  && {\Sh(X)} \ar@/_3ex/[ll]_{\Lambda}},
\]
where the sections functor $\Gamma$ sends a continuous map $f:Y\ra X$ to the sheaf
\[
\Gamma(f)(U)=\{ \;\text{continuous sections}\; s: U\ra Y \;\text{of}\; f\,\}\,,
\]
and its left adjoint assigns to a sheaf $F$ its bundle of germs $\Lambda(F)\,$.
The unit $F\ra \Gamma\Lambda(F)$ is an isomorphism for any sheaf $F$,
while the counit $\Lambda\Gamma(f)\ra f$ for a continuous map $f$
is an isomorphism just when $f$ is a local homeomorphism, also called an \'etale map.
Likewise, for a left involutive semigroup $S$ we shall define an adjointness
\[
\xymatrix{\text{Semigroups}/S \ar@/_3ex/[rr]_{\Gamma} \ar@{}[rr]|{\bot}
  && {\B(S)} \ar@/_3ex/[ll]_{\Lambda}}
\]
whose unit $P\ra \Gamma\Lambda(P)$ is an isomorphism for any object $P$ of $\B(S)\,$,
and whose counit $\Lambda\Gamma(f)\ra f$ at a $*$-homomorphism
$f:X\ra S$ is an isomorphism if and only if $X$ is a left involutive semigroup
and $f$ is \'etale (Def.\ \ref{def:etale} and Prop.\ \ref{prop:lift}).

Let us say that a $*$-semigroup is a semigroup $X$
equipped with a unary operation $*:X\ra X$ such that
$x^{**}=x$ and $xx^*x=x$ for all $x\in X\,$.
In order to spell out the details of the adjointness $\Lambda\adj \Gamma$
we shall develop some basic properties of $*$-semigroups which might be of independent interest.
In the literature (cf.\ \cite{NS})
a regular $*$-semigroup is defined as a $*$-semigroup for which $(xy)^*=y^*x^*$ holds for all
$x,y\in X\,$.
For consistency, the latter are herein called \emph{involutive semigroups\/}, cf.\  \cite{BB}.
The more general relation $(xy)^*=(x^*xy)^*x^*$ defines what we call a
\emph{left involutive semigroup\/}, introduced by the second-named author in \cite{F7}.

We shall say that a $*$-semigroup is birestrictive if it satisfies
$\dd(xy)\leq\dd(y)$ and $\cc(xy)\leq\cc(x)\,$, where $\dd(-)$ and $\cc(-)$ are the domain
and codomain operators.
By definition, {\em a quasi-involutive semigroup\/}
is a birestrictive, locally involutive semigroup (Def.\ \ref{def:invsemigroup}).
Any inverse semigroup is involutive, and any involutive semigroup is quasi-involutive.
We establish an \emph{Ehresmann-Schein-Nambooripad correspondence\/}
between quasi-involutive semigroups and ordered groupoids
with mediator (cf.\  Thm.\ \ref{thm:ESN}), generalising the classical correspondence between
inverse semigroups and inductive groupoids (cf.\ \cite{L}), as well as the correspondence between
involutive semigroups and chained projection groupoids, recently established in \cite{EA}.

Left involutive semigroups are locally involutive and corestrictive (cf.\  Lem.\ \ref{lem:birestrictive}),
but in general not birestrictive, so that our ESN-correspondence does not apply to them.
Nonetheless, for any \'etale $*$-homomorphism of left involutive semigroups $f:X\ra S$
there exists a uniquely determined presheaf $P_f$ on the left cancellable category $L(S)$ such that $\Lambda(P_f)=X\,$.
According to Thm.\ \ref{thm:BSleft} the resulting functor $\Lambda$
admits a right adjoint $\Gamma$ provided the morphisms in the category of semigroups
over $S$ are defined correctly by means of so-called left $*$-homomorphisms.

We use Thm.\ \ref{thm:BSleft} to establish a new presentation of $\B(S)$
as a category of what we call balanced $S$-sets (Def.\ \ref{def:bal} and Cor.\ \ref{cor:bal}).
In turn, this result helps us show in \S~\ref{sec:inv} that if $X$ is a left involutive semigroup,
then an \'etale $*$-homomorphism $X\ra S$
can be embedded in an involutive semigroup $Y$ and a $*$-homomorphism $Y\ra S$
as in the following diagram.
\[
\xymatrix{X \ar[rr] \ar[dr] && Y \ar[dl]\\
& S}
\]
The embedding $X\ra Y$ is an injective left $*$-homomorphism.
In effect, the left involutive semigroup $X$ can be made involutive,
but at the cost of giving up the \'etale nature of the structure map.
We introduce and develop what we call $S$-modules and algebras in
order to provide a conceptual explanation.

In summary, the classes of semigroups we work with are (in increasing generality) inverse semigroups,
involutive semigroups, left/right involutive semigroups, locally involutive semigroups, and
finally $*$-semigroups.
The class of quasi-involutive semigroups lies between involutive semigroups
and locally involutive semigroups,
but is different from the class of left/right involutive semigroups.
The natural class of morphisms in all cases are $*$-homomorphisms,
i.e., those morphisms preserving involution and binary product.
The inclusion of each class into another is then full and faithful.
It turns out that it is sometimes necessary when working in the slice category
over a \emph{left} involutive semigroup
to consider the larger class of \emph{left} $*$-homomorphisms.
\begin{figure}\caption{A map of terminology}
\[
\xymatrix{
&& \text{left involutive}  \ar@{^{(}->}[dr]  \\
\text{inverse} \ar@{^{(}->}[r] & \text{involutive}  \ar@{^{(}->}[ur]  \ar@{^{(}->}[dr]  \ar@{^{(}->}[r]
&   \text{quasi-involutive} \ar@{^{(}->}[r] & \text{locally involutive}\\
&& \text{right involutive}  \ar@{^{(}->}[ur] }
\]
\end{figure}

\begin{svgraybox}

\noindent\emph{Acknowledgements:}

Many thanks to the two referees who provided some valuable feedback leading to
a significant improvement and revision of the original manuscript.
The first-named author would like to thank
C\'elia Borlido and Mai Gehrke for several enlightening discussions on semigroups and related topics.

\end{svgraybox}

\section{Locally involutive semigroups}
\vspace{2ex}

\subsection{Terminology and basic properties}
\leavevmode\par\vspace{1ex}

An {\em involutive set\/} is a set $X$ equipped with an idempotent endomorphism,
throughout denoted $x\mapsto x^*\,$ so that $x^{**}=x$ for $x\in X\,$.

{\em A $*$-semigroup\/} is an involutive set $X$ equipped with an associative binary operation
$X\times X\ra X:(x,y)\mapsto xy$ such that $xx^*x=x$ (i.e.\ $x$ is {\em a partial isometry}) for all $x\in X\,$.
{\em A $*$-morphism\/} $f:X\ra S$ of involutive sets is a map $f:X\ra S$
such that $f(x^*)=f(x)^*$ for all $x\in X\,$.
A {\em $*$-homomorphism}
$f:X\ra S$ of $*$-semigroups is $*$-morphism
of the underlying involutive sets such that $f(xy)=f(x)f(y)\,$.

An element $x$ of a $*$-semigroup is {\em self-adjoint\/} if $x^*=x\,$, and {\em idempotent} if $x^2=x\,$.
Any self-adjoint, idempotent element is called {\em a projection}.
The set $E(X)$ of idempotent elements of a $*$-semigroup $X$
is partially ordered by $e\leq f$ if $e=ef=fe\,$.
We get a restricted partial order on the set $P(X)$ of projections of $X\,$.
For each element $x$ of a $*$-semigroup $X$ the product $x^*x$
is called the {\em domain} of $x$, and the product $xx^*$ is called the {\em codomain} of $x\,$.
We write $\dd(x)$ and $\cc(x)$ for domain and codomain of $x\,$,
yielding the usual absorption laws $\cc(x)x=x=x\dd(x)$ for any $x\in X\,$.
Domain and codomain of any $x\in X$ are idempotent, but in general not self-adjoint.

A $*$-semigroup $X$ is called {\em restrictive\/}
if $\dd(xy)\leq\dd(y)\,$, and {\em corestrictive\/} if $\cc(xy)\leq\cc(x)\,$.
It is called {\em birestrictive\/} if it is both restrictive and corestrictive.

{\em An involutive semigroup\/} is a $*$-semigroup such that the relation
$(xy)^*=y^*x^*$ holds for all $x,y\in X\,$.
In an involutive semigroup domain and codomain are projections.
Involutive semigroups have been introduced in \cite{NS}
under the name of {\em regular $*$-semigroups\/}, and shown in \cite{EA}
to correspond to a certain class of ordered groupoids,
extending the classical ESN-correspondence between inverse semigroups and inductive groupoids.
Involutive semigroups are birestrictive.

The objects of the {\em ordered groupoid\/} $\sG(X)$ associated with an involutive semigroup $X$
are the projections of $X\,$, and its morphisms are the elements $x\in X\,$,
where domain and codomain are given by $\dd(x)=x^*x$ and $\cc(x)=xx^*\,$,
the composition law by $x\circ y=xy\,$,
and the two-sided inverse of $x:\dd(x)\ra\cc(x)$ by $x^*:\dd(x^*)\ra\cc(x^*)\,$.

We have the following three generalisations of involutive semigroups, all having an associated ordered groupoid.

\begin{definition}\label{def:invsemigroup}
A {\em left involutive semigroup\/} is a $*$-semigroup such that $(xy)^*=(\dd(x)y)^*x^*$ for all $x,y\,$.
A {\em right involutive semigroup\/} is a $*$-semigroup such that $(xy)^*=y^*(x\cc(y))^*$ for all $x,y$.
A {\em locally involutive semigroup\/} is a $*$-semigroup such that the identity
$(xy)^*=y^*x^*$ holds whenever
$\dd(x)=\cc(y)$ or $x$ is a projection such that $x\leq\cc(y)$ or $y$ is a projection such that $y\leq\dd(x)\,$.
\end{definition}

A {\em quasi-involutive semigroup\/} is a birestrictive, locally involutive semigroup.

In a locally involutive semigroup, domain and codomain are projections.
Moreover, being a {\em left projection\/} (i.e.\ a fixpoint for $\dd$)
or being a {\em right projection} (i.e.\ a fixpoint for $\cc$) amounts to being a projection.

\begin{lemma}\label{lem:reduct}
Any involutive semigroup is left and right involutive.
Any left (or right) involutive semigroup is locally involutive.
\end{lemma}
\begin{proof}
The first assertion is immediate.
For the second, if $\cc(x)=\dd(y)\,$,
then in a left involutive semigroup we get $(xy)^*=(\cc(y)y)^*x^*=y^*x^*$
while in a right involutive semigroup we get $(xy)^*=y^*(x\cc(x))^*=y^*x^*\,$.
Moreover, in a left involutive semigroup, if $y$ is a projection bounded by $\dd(x)\,$,
then $(xy)^*=(\dd(x)y)^*x^*=y^*x^*\,$,
while if $x$ is a projection bounded by $\cc(y)\,$,
then $(y^*x)^*=(\dd(y^*)x)^*y=(\cc(y)x)^*y=xy$ so that $(xy)^*=y^*x^*\,$.
For a right involutive semigroup the proof is dual.
\end{proof}

Projections in a left or right involutive semigroup behave in a similar way
as projections in a Jones' left or right \emph{$P$-Ehresmann semigroup\/} (cf.\ \cite{J})
insofar as they do not generally commute,
but rather satisfy identities like idempotents in a left or right regular band provided
they `compose' to projections.

\begin{lemma}\label{lem:birestrictive}
Any left involutive semigroup is corestrictive.
Any right involutive semigroup is restrictive.
Any involutive semigroup is birestrictive, and hence quasi-involutive.
For projections $p,q$ such that $pq$ is also a projection,
we have
\[
pqp=pq\;\text{\em (resp.\ }pqp=qp\,,\;\text{\em resp.\ }pq=qp\text{)}
\]
in any left involutive ({\em resp.\ right involutive, resp.\ involutive)} semigroup.
\end{lemma}
\begin{proof}
The statements for left and right involutive semigroups are dual.
For general $*$-semigroups we have
\[
\cc(x)\cc(xy)=xx^*xy(xy)^*=xy(xy)^*=\cc(xy)\,,
\]
while for left involutive semigroup we have
\[
\cc(xy)\cc(x)=xy(\dd(x)y)^*x^*xx^*=xy(\dd(x)y)^*x^*=\cc(xy)
\]
so that $\cc(xy)\leq \cc(x)\,$.
The second assertion is obvious.
The last assertion follows from the fact that projections, left projections and right projections coincide
in any locally involutive semigroup,
and hence in any left or right involutive semigroup by Lem.\ \ref{lem:reduct}.
\end{proof}

\begin{lemma}\label{lem:leftright}
In any $*$-semigroup $X\,$, for $x,y\in X$ and $p,q\in P(X)$ we have
\begin{enumerate}
\item{$\cc(y)\leq p$ amounts to the identities $py=y$ and $y^*p=y^*\,$;}
\item{$\dd(x)\leq q$ amounts to the identities $xq=x$ and $qx^*=x^*\,$;}
\item{$\cc(y)\leq\dd(x)$ amounts to the identities $\dd(x)y=y$ and $y^*\dd(x)=y^*\,$;}
\item{$\cc(y)\geq\dd(x)$ amounts to the identities $x\cc(y)=x$ and $\cc(y)x^*=x^*\,$;}
\end{enumerate}
In a left involutive semigroup $py=y$ implies $y^*p=y^*\,$.
In a right involutive semigroup $xq=x$ implies $qx^*=x^*\,$.
\end{lemma}

\begin{corollary}\label{cor:3cond}
If $x,y$ are elements of a left involutive semigroup,
then $\cc(y)\leq\dd(x)$ if and only if $\dd(x)y=y\,$,
and $\cc(y)\geq\dd(x)$ if and only if $\cc(y)x^*=x^*\,$.
\end{corollary}

\begin{remark}\label{rem:three}
Picture the three conditions $y=\dd(x)y\,$, $\dd(x)=\cc(y)\,$, and $\cc(y)x^*=x^*$
in a left involutive semigroup as follows (cf.\  Cor.\ \ref{cor:3cond}).
\[
\xymatrix{\cdot \ar[r]^-y & \cdot \ar@{}[d]|-{\leq}\\
 & \cdot \ar[r]^-x & \cdot}\;\;\;
\xymatrix{\cdot \ar[r]^-y & \cdot \ar[r]^-x & \cdot}\;\;\;
\xymatrix{& \cdot \ar[r]^-x \ar@{}[d]|-{\leq} & \\
\cdot \ar[r]^-y & \cdot}
\]
\end{remark}

\subsection{Partial orderings of a $*$-semigroup}\label{subsec:po}
\leavevmode\par\vspace{1ex}

Sometimes it is productive to argue using a partial ordering of {\em all} elements of a $*$-semigroup,
not just a partial ordering of its projections.
We define here two partial orders:
a \emph{left order} $\leq_l$ and a \emph{right order} $\leq_r$
available for any $*$-semigroup.
They are dual to each other and well adapted to left and right involutive semigroups.
We show in Lem.\ \ref{lem:po}(7) below that for locally involutive semigroups
the two partial orders coincide so that we shall simply write $\leq$ in this case,
and refer to it as the \emph{natural partial order} for locally involutive semigroups.
This coincidence of left and right partial orders is a crucial ingredient
in defining the ordered groupoid associated to a locally involutive semigroup.\vspace{1ex}

\begin{definition}
We declare $x\leq_l y$ if $\dd(x)\leq\dd(y)\text{ and }x=y\dd(x)\,$,
and $x\leq_r y$ if $\cc(x)\leq\cc(y)\text{ and }x=\cc(x)y\,$.
\end{definition}

\begin{remark}
Obviously $x\leq_l x\,$.
If $x\leq_l y$ and $y\leq_l x\,$,
then $\dd(x)=\dd(y)$ and $x=y\dd(x)=y\dd(y)=y\,$.
If $x\leq_l y\leq_l z\,$, then $\dd(x)\leq_l\dd(y)\leq_l\dd(z)$
and $x=y\dd(x)=z\dd(y)\dd(x)=z\dd(x)\,$, thus $x\leq_l z\,$.
Therefore $\leq_l$ is a partial order relation. A dual argument applies to $\leq_r\,$.
\end{remark}

\begin{lemma}\label{lem:po}
For any $*$-semigroup $X$ and any $x,y\in X$ we have
\begin{enumerate}
\item{$x\leq_l y$ if and only if $x\dd(y)=x,\,y^*x=\dd(x)\text{ and }yx^*=\cc(x)\,$;}
\item{$x\leq_r y$ if and only if $\cc(y)x=x,\,x^*y=\dd(x)\text{ and }xy^*=\cc(x)\,$;}
\item{If $X$ is left involutive, then $x\leq_ly$ if and only if $y^*x=\dd(x)\text{ and }yx^*=\cc(x)\,$;}
\item{If $X$ is right involutive, then $x\leq_ry$ if and only if $x^*y=\dd(x)\text{ and }xy^*=\cc(x)\,$;}
\item{$x\leq_ly\text{ and }xy^*=\cc(x)$ together imply $x\leq_ry\,$;}
\item{$x\leq_ry\text{ and }y^*x=\dd(x)$ together imply $x\leq_ly\,$;}
\item{If $X$ is locally involutive, then left and right partial orders coincide;}
\item{If $X$ is locally involutive, then $x\leq_l y$ if and only if $x^*\leq_r y^*\,$.}
\end{enumerate}
\end{lemma}
\begin{proof}
1. In a $*$-semigroup, $\dd(x)\dd(y)=\dd(x)$ holds if and only if $x\dd(y)=x\,$,
and $y\dd(x)=x$ holds if and only if $yx^*=\cc(x)\,$.
Moreover, if $\dd(y)\dd(x)=\dd(x)\,$, then under the hypothesis $y\dd(x)=x$ we have $y^*x=\dd(x)\,$.
Conversely, since in any $*$-semigroup $\dd(y)y^*x=y^*x$ holds,
the identity $y^*x=\dd(x)$ implies $\dd(y)\dd(x)=\dd(x)\,$.

2 is dual to 1.

3. In a left involutive semigroup $\dd(y)\dd(x)=\dd(x)$ implies $\dd(x)\dd(y)=\dd(x)$
because of Lem.\ \ref{lem:leftright} so that the identity $x\dd(y)=x$ in 1 is redundant.

4 is dual to 3.

5 and 6.
The two implications are dual.
For 5.\ it suffices to establish $x^*y=\dd(x)$ and $\cc(y)x=x$ by 2.
It follows from $xy^*=\cc(x)$ that $\dd(x)y^*=x^*\,$.
Substituting we get $x^*y=\dd(x)y^*y=\dd(x)\dd(y)=\dd(x)$ establishing the first identity.
Moreover, we have $\cc(y)x=yy^*x=y\dd(x)=x$ establishing the second identity.

7. In any locally involutive semigroup, if $x\leq_ly$ and thus $\dd(x)\leq\cc(y^*)$ and $x=y\dd(x)\,$,
then $(\dd(x)y^*)^*=y\dd(x)=x$ so that $xy^*=x\dd(x)y^*=xx^*=\cc(x)\,$.
Then $x\leq_r y$ follows from 5.
The proof that $x\leq_ry$ implies $x\leq_l y$ is dual, using 6.

8. It suffices to show that $x\leq_ly$ implies $x^*\leq_ry^*\,$.
If $\dd(x)\leq\dd(y)\,$, then $\cc(x^*)\leq\cc(y^*)\,$.
Moreover, the identity $x=y\dd(x)$ implies $x^*=(y\dd(x))^*=\dd(x)y^*=\cc(x^*)y^*$
so that $x^*\leq_ry^*$ as required.
\end{proof}

\begin{remark}
In \cite{F7} the second-named author formulates the partial order on left involutive semigroups
as characterization 3 of Lem.\ \ref{lem:po}.
We ought to mention that this formulation was known to Drazin almost thirty years earlier \cite{Dr}.
Since left involutive semigroups are locally involutive, the left order coincides with the right order.
However, characterization 3 of this partial order is valid only for left involutive semigroups.
For right involutive semigroups the partial order relation is given by characterization 4,
while for general locally involutive semigroups three identities are necessary to characterize the partial order.
\end{remark}


\subsection{ESN-correspondence for quasi-involutive semigroups}\label{subsec:og}
\leavevmode\par\vspace{1ex}

Just like in the involutive case a locally involutive semigroup $X$ defines a groupoid $\sG(X)$ with object-set $\sG(X)_0$
the set $P(X)$ of projections of $X\,$,
and with morphism-set $\sG(X)_1$ the set $X\,$.
The domain and codomain maps are $\dd$ and $\cc\,$.
Composition is defined $x\circ y=xy$ just as in the involutive case.
In particular, for composable morphisms $x,y\in\sG(X)_1$ we have $\dd(xy)=\dd(y)$ and $\cc(xy)=\cc(x)\,$.
Moreover, the projections play a double role: they form the objects of the groupoid $\sG(X)$
as well as the identities at these objects.

\begin{proposition}\label{prop:groupoid}
For any locally involutive semigroup $X\,$,
the groupoid $\sG(X)$ is an ordered groupoid with respect to the natural partial ordering of $X=\sG(X)_1\,$.
Moreover, restriction and corestriction in $\sG(X)$ are given by pre and postcomposition
with the corresponding projection.
\end{proposition}
\begin{proof}
We first show that the partial order is compatible with composition.
Indeed, if $x_1\leq x_2$ and $y_1\leq y_2$ such that $\dd(x_1)=\cc(y_1)$ and $\dd(x_2)=\cc(y_2)\,$,
then
\[
x_1y_1=x_2\dd(x_1)y_1=x_2\cc(y_1)y_1=x_2y_1=x_2y_2\dd(y_1)=x_2y_2\dd(x_1y_1)\,,
\]
and thus $x_1y_1=x_2y_2\dd(x_1y_1)\,$, i.e.\ $x_1y_1\leq x_2y_2\,$.
By Lem.\ \ref{lem:po}(8) the partial order is compatible with inversion in $\sG(X)\,$.
It remains to be shown that for any $x:\dd(x)\ra\cc(x)$ and any projection $p$ bounded by $\dd(x)$
there exists a unique `restriction' $x_{\downharpoonleft p}$ in $\sG(X)$
such that $\dd(x\!\downharpoonleft_p)=p$ and $x\!\downharpoonleft_p\leq x\,$.
We claim that $xp$ has the required properties.
Indeed, $\dd(xp)=(xp)^*xp=p^*x^*xp=p^*p=p\,$.
Moreover, there is no choice: $x\!\downharpoonleft_p\leq x$ implies
$x\!\downharpoonleft_p=x\dd(x\!\downharpoonleft_p)=xp\,$.
Dually, the corestriction $x\!\upharpoonleft^q$ of $x$ to
a projection $q$ bounded by $\cc(x)$ is given by postcomposition $qx\,$.
Restriction and corestriction are related
by the formula $x\!\upharpoonleft^q=(x^{-1}\!\downharpoonleft_q)^{-1}\,$.
\end{proof}

\begin{remark}\label{rem:extended}
We may define an `extended' composition in the ordered groupoid $\sG(X)\,$.
If $\cc(y)\leq\dd(x)$ let
\[
x\otimes y = (x\!\downharpoonleft_{\cc(y)})y\,,
\]
and if $\dd(x)\leq\cc(y)$ let
\[
x\otimes y=x(y\!\upharpoonleft^{\dd(x)})\,.
\]
This extended composition is associative whenever it is defined
because of the uniqueness of restriction/corestriction in $\sG(X)\,$.
Moreover, it reflects multiplication in $X$
because in the first case $x\otimes y= x\cc(y)y=xy$ and in the second case $x\otimes y=x\dd(x)y=xy\,$.
Extended composition is how composition is defined in the left cancellative category
associated with an ordered groupoid.
\end{remark}

Nonetheless, the ordered groupoid $\sG(X)$ as such does not determine the multiplication
of a locally involutive semigroup $X\,$.
We therefore describe an additional structure on $\sG(X)\,$,
which will enable us to recover the underlying locally involutive semigroup $X\,$.

\begin{definition}\label{def:mediator}
\emph{A mediator} for an ordered groupoid $\G$ is an order-preserving map
\[
\mm:\G_0\times\G_0\ra\G_1\,;\;(p,q)\mapsto\mm_{pq}
\]
such that $\dd(\mm_{pq})\leq q\text{ and }\cc(\mm_{pq})\leq p\,$.
We also require that the induced left action
\[
\G_0\times\G_1\ra\G_1\,;\;(p,x)\mapsto {}^px=\mm_{p\cc(x)}\otimes x
\]
and induced right action
\[
\G_1\times\G_0\ra\G_1\,;\;(x,q)\mapsto x^q=x\otimes \mm_{\dd(x)q}
\]
commute meaning that $({}^px)^q={}^p(x^q)\,$.
\end{definition}

A morphism of ordered groupoids with mediator is a morphism of the underlying ordered groupoids preserving the mediator.

\begin{proposition}
The ordered groupoid $\sG(X)$ of a quasi-involutive semigroup $X$ has a canonical mediator $\mm_{pq}=pq\,$.
This defines a faithful functor from the category of quasi-involutive semigroups
to the category of ordered groupoids with mediator.
\end{proposition}
\begin{proof}
That $\mm_{pq}=pq$ is order-preserving
follows from the compatibility between multiplication and partial ordering of $X\,$.
The domain and codomain conditions on $\mm$ are consequences of the hypothesis that $X$ is birestrictive.
The commutation of induced left and right actions is immediate because
${}^px=p\dd(x)x=px$ and $x^q=x\cc(x)q=xq\,$.
This functor is faithful because $X=\sG(X)_1\,$.
\end{proof}

\begin{proposition}
Any ordered groupoid $\G$ with mediator $\mm$ induces
a quasi-involutive semigroup $\eS(\G)=\G_1$
with multiplication $xy=x\otimes\mm_{\dd(x)\cc(y)}\otimes y$
and with involution induced by inversion in $\G\,$.
The projections of $\eS(\G)$ are the identities of $\G\,$.
This defines a faithful functor from the category of ordered groupoids
with mediator to the category of quasi-involutive semigroups.
\end{proposition}
\begin{proof}
For the associativity of the multiplication observe that
\[
x(yz)=x\otimes{}^{\dd(x)}(y^{\cc(z)}\otimes z)=x\otimes{}^{\dd(x)}(y^{\cc(z)})\otimes z=x\otimes({}^{\dd(x)}y)^{\cc(z)}\otimes z
\]
\[
=(x\otimes{}^{\dd(x)}y)^{\cc(z)}\otimes z=(xy)z\,.
\]
Since domain and codomain operations of $\eS(\G)$ coincide with those of $\G$
it follows that $(xy)^*=y^*x^*$ whenever $\dd(x)=\cc(y)\,$.
Moreover, in the special case where $x$ is a projection bounded by $\cc(y)$ or $y$
is a projection bounded by $\dd(x)\,$, $*$-commutativity follows from the fact
that restriction/corestriction in $\G$ are exchanged under inversion.
Therefore, the semigroup $\eS(\G)$ is locally involutive.
Moreover the projections of $\eS(\G)$ are precisely the identities of $\G$
because in an ordered groupoid there are no automorphisms of order $2\,$.
Finally, $\eS(\G)$ is birestrictive because domain, resp.\ codomain,
of $xy=x\otimes\mm_{\dd(x)\cc(y)}\otimes y$ are bounded by $\dd(y)\,$, resp.\ $\cc(x)\,$.

Any morphism of ordered groupoids preserves extended composition and inversion
so that a morphism preserving the mediator must also preserve multiplication
and involution of the induced semigroups.
\end{proof}

The following proof of our ESN-correspondence has been inspired by Lawson's proof for Ehresmann semigroups, cf.\  \cite{L21}.

\begin{svgraybox}
\begin{theorem}[ESN-correspondence]\label{thm:ESN}
The functors $\eS$ and $\sG$ induce an isomorphism between the categories
of quasi-involutive semigroups and of ordered groupoids with mediator.
This isomorphism restricts to the ESN-correspondence
between inverse semigroups and ordered groupoids
with trivial mediator $\mm_{pq}=1_{p\wedge q}$ {\em (aka inductive groupoids)},
and to a correspondence between involutive semigroups and ordered groupoids with symmetric mediator
{\em (i.e.\ $\mm_{pq}^{-1}=\mm_{qp}\,$)}.
\end{theorem}
\end{svgraybox}
\begin{proof}
It follows from Rem.\ \ref{rem:extended} that $\sG\circ\eS$ is the identity functor
on the category of quasi-involutive semigroups.
Since $\sG$ takes the objects of an ordered groupoid with mediator bijectiviely
to the projections of the associated semigroup, the composite $\eS\circ\sG$ is also the identity functor.

For an ordered groupoid with trivial mediator the corresponding semigroup has the property
that for each element $x$ the chosen $x^*$ (corresponding to $x^{-1}$ in the groupoid)
is the \emph{unique} quasi-inverse of $x$ because in the formula for the multiplication
there are only restriction/corestriction and compositions, whence the semigroup is inverse.
Conversely, if the semigroup is inverse,
then its projections commute (and the product of two projections is a projection)
so that the mediator is trivial.

For an involutive semigroup we have $(pq)^*=q^*p^*=qp\,$,
which reads $\mm_{pq}^{-1}=\mm_{qp}\,$,
so that the associated mediator is symmetric.
For an ordered groupoid with symmetric mediator the multiplication of the associated semigroup satisfies
\[
(xy)^*=(x\otimes\mm_{\dd(x),\cc(y)}\otimes y)^{-1}
=y^{-1}\otimes\mm_{\dd(x),\cc(y)}^{-1}\otimes x^{-1}
\]
\[
=y^{-1}\otimes\mm_{\dd(y^{-1}),\cc(x^{-1})}\otimes x^{-1}=y^*x^*
\]
so that we get an involutive semigroup.
\end{proof}

\begin{proposition}[cf.\ \cite{BB}]\label{prop:commproj}
A semigroup is inverse if and only if it is quasi-involutive and has commuting projections.
In particular, any left (or right) involutive semigroup with commuting projections is inverse.
\end{proposition}

\begin{proof}
Inverse semigroups are involutive and hence quasi-involutive, and their projections commute.
Conversely, if the projections of a $*$-semigroup $X$ commute,
then they compose so that the set $P(X)$ of projections is a semilattice.
This implies that if $X$ is moreover quasi-involutive,
then its groupoid $\sG(X)$ has trivial mediator so that $X$ is inverse.
For the second statement note that by Lem.\ \ref{lem:birestrictive} any left involutive semigroup is corestrictive.
If the projections commute, then it is also restrictive and hence quasi-involutive.
Indeed, we always have $\dd(xy)\dd(y)=\dd(xy)\,$.
Under commutation of projections, this implies $\dd(y)\dd(xy)=\dd(xy)$ whence $\dd(xy)\leq\dd(y)\,$.
\end{proof}

\section{Involutive and balanced sets over $*$-semigroups}
\vspace{2ex}

\subsection{Left $*$-homomorphisms}
\leavevmode\par\vspace{1ex}

Let us say that a $*$-morphism $f$ of $*$-semigroups is {\em a left $*$-homomorphism\/} if
$f(xy)=f(x)f(\dd(x)y)$
holds for all $x,y\in X\,$.

\begin{remark}\label{rem:preproj}
If $f$ is a left $*$-homomorphism, then
\[
f(\dd(x))=f(x^*)f(xx^*x)=f(x)^*f(x)=\dd(f(x))\text{ and }
\]
\[
f(\cc(x))=f(x)f(x^*xx^*)=f(x)f(x)^*=\cc(f(x))
\]
so that left $*$-homomorphisms preserve left and right projections.
A left $*$-homomorphism also preserves the left partial order (\cite{F7}, Lem.\ 2.30).
\end{remark}

\begin{exercise}\label{ex:leftstar}
Show that any $*$-homomorphism is a left $*$-homomorphism.
Show that the composite of two left $*$-homomorphisms is a left $*$-homomorphism.
Show that for a surjective left $*$-homomorphism $f:X\ra Y\,$, if $X$ is a left involutive semigroup, then so is $Y\,$.
\end{exercise}

\begin{remark}
Any homomorphism of inverse semigroups is a $*$-homomorphism.
However, since typically we are working with semigroups that are not inverse we have to check both properties:
preservation of the involution and preservation of the multiplicative structure.

Note also that the set-theoretical inverse of a bijective $*$-homomorphism is a $*$-homomorphism, while
the same property does in general not hold for bijective left $*$-homomorphisms, cf.\  Rem.\ \ref{rem:comparison}.
\end{remark}

\begin{remark}\label{rem:fxy}
If $f:X\ra Y$ is a left $*$-homomorphism,
and $y=\dd(x)y\,$, then $f(xy)=f(x)f(y)$.
\end{remark}

\begin{lemma}\label{lem:fdt}
A left $*$-homomorphism $f$ of left involutive semigroups is a $*$-homomorphism
if and only if $f(px)=f(p)f(x)$ for any projection $p$ and element $x\,$.
\end{lemma}
\begin{proof}
Suppose that a left $*$-homomorphism $f$ has the stated property.
Then we have
\[
f(xy)=f(x)f(\dd(x)y)=f(x)f(\dd(x))f(y)=f(x)\dd(f(x))f(y)=f(x)f(y)\,.
\]
\end{proof}

\begin{definition}\label{def:etale}
We shall say that a left $*$-homomorphism $f:X\ra S$ of $*$-semigroups
is {\em \'etale\/} if for every $x\in X$ the restriction
$f:xX\ra f(x)S$ is a bijection,
where $xX=\{\, xy\mid y\in X\,\}\,$.
\end{definition}

\begin{exercise}\label{ex:fg}
Let $f$ and $g$ be composable left $*$-homomorphisms.
Show that if $f$ and $fg$ are \'etale then so is $g\,$.
\end{exercise}

\begin{remark}\label{rem:coset}
We introduce the notation $X\downarrow p=\{x\in X\,|\,\cc(x)\leq p\}$ for any $p\in P(X)\,$.
Note that if $X$ is a left involutive semigroup,
then $X\downarrow\cc(x)=xX\,$, and we have
$xX\subseteq yX$ if and only if $x\in yX$ if and only if $\cc(x)\leq\cc(y)$
if and only if $x=\cc(y)x\,$.
\end{remark}

\begin{proposition}\label{prop:lift}
For a left $*$-homomorphism $f:X\ra S$ of left involutive semigroups the following three properties are equivalent:
\begin{enumerate}
\item{$f$ is \'etale;}
\item{the restriction $f:X\downarrow p\ra S\downarrow f(p)$ is a bijection for all $p\in P(X)\,$;}
\item{for a projection $p\in P(X)$, any equation $s=f(p)s$ in $S\,$ lifts uniquely to an equation $x=px$ in $X\,$.}
\end{enumerate}
\end{proposition}
\begin{proof}
By Lem.\ \ref{lem:leftright}, in any left involutive semigroup the identity $px=x$
amounts to the identities $p\cc(x)=\cc(x)=\cc(x)p\,$,
i.e.\ to the relation $\cc(x)\leq p\,$.
This shows the equivalence between 2 and 3 while Rem.\ \ref{rem:coset}
shows the equivalence between 1 and 3.
\end{proof}

\begin{lemma}\label{lem:reflect}
An \'etale left $*$-homomorphism reflects right projections.
\end{lemma}
\begin{proof}
Let $f:X\ra S$ be an \'etale left $*$-homomorphism of $*$-semigroups.
Suppose that $f(x)$ is a right projection,
so that $f(x)=\cc(f(x))=f(\cc(x))$ (Rem.\ \ref{rem:preproj}).
We have $x,\cc(x)\in xX\,$, so $x=\cc(x)$ by the \'etale property, whence $x$ is a right projection.
\end{proof}

\begin{proposition}\label{prop:starhomo}
Let $X,Y$ and $S$ be $*$-semigroups.
In the diagram below suppose $f$ is a $*$-homomorphism,
$h$ is an \'etale $*$-homomorphism,
and $\psi$ is a left $*$-homomorphism.
\[
\xymatrix{X \ar[dr]_-{f} \ar[rr]^-{\psi}
&& Y \ar[dl]^-{h}\\
& S}
\]
Then $\psi$ is a $*$-homomorphism. Moreover, if $f$ is \'etale, then so is $\psi\,$.
\end{proposition}
\begin{proof}
Let $x,y\in X\,$.
Since $\psi(xy)=\psi(x)\psi(\dd(x)y)$
the two elements $\psi(xy)$ and $\psi(x)\psi(y)$ are both members
of $\psi(x)Y\,$.
Moreover, we have
\[
h(\psi(xy))=f(xy)=f(x)f(y)=h\psi(x)h\psi(y)=h(\psi(x)\psi(y))
\]
so that $\psi(xy)=\psi(x)\psi(y)$ because $h$ is \'etale.
The second assertion is Ex.\ \ref{ex:fg}.
\end{proof}

\subsection{Involutive $S$-sets}\label{subsec:action}
\leavevmode\par\vspace{1ex}

\begin{definition}\label{def:action}
Let $f:X\ra S$ be an \'etale left $*$-homomorphism of $*$-semigroups.
Then we may define a right action of $S$ on $X\,$:
for any $x\in X$ and $s\in S\,$, let $xs$ denote the unique element of $X$ such that
$\cc(x)(xs)=xs$ and $f(xs)=f(x)s\,$.
\end{definition}

This is well defined since fixpoints of $X$ under left multiplication by $\cc(x)$
are precisely the elements of $xX$ which by assumption are taken bijectively to $f(x)S\,$.
We call this action {\em the canonical action of $S$ on $X$ associated with $f$\/}.
This right action preserves the codomain.

\begin{exercise}
Verify that the action defined in Def.\ \ref{def:action} is associative: $(xs)t= x(st)\,$.
\end{exercise}

\begin{lemma}\label{lem:xfy}
Let $f:X\ra S$ be an \'etale left $*$-homomorphism of $*$-semigroups.
Then for any $x,y\in X$ we have $xf(\dd(x)y)=xy\,$ and in particular $xf(\dd(x))=x\,$.
Moreover, $f$ is a $*$-homomorphism if and only if $xf(y)=xy$ for all $x,y\in X\,$, in which case
there is a mixed associativity $(xy)s=x(ys)\,$ for all $s\in S\,$.
\end{lemma}
\begin{proof}
By definition, $f(xf(\dd(x)y))=f(x)f(\dd(x)y)=f(xy)\,$.
Since $xy$ is a left $\cc(x)$-fixpoint we get  $xf(\dd(x)y)=xy\,$ by \'etaleness of $f$.
For $y=\dd(x)$ this yields $xf(\dd(x))=x\,$.
If $f$ is a $*$-homomorphism then $xy=xf(\dd(x)y)=xf(\dd(x))f(y)=xf(y)\,$.
Conversely, if $xy=xf(y)$ then $f(xy)=f(xf(y))=f(x)f(y)\,$.
For the last assertion we have $(xy)s=(xf(y))s=x(f(y)s)=xf(ys)=x(ys)\,$.
\end{proof}

\begin{definition}\label{def:Sset}
Let $S$ be a $*$-semigroup.
{\em An involutive $S$-set\/} is an involutive set $X\,$,
a $*$-morphism $f:X\ra S\,$, and an associative right action by $S$ in $X\,$,
written $xs\,$,
such that for all $x\in X$ and $s\in S$ we have:
\begin{enumerate}
\item{$f(xs)=f(x)s$ (equivariance);}
\item{$x\dd(f(x))=x$ (unit).}
\end{enumerate}
\end{definition}

\begin{proposition}\label{prop:Sset}
Let $S$ be a $*$-semigroup,
and $f:X\ra S$ be an involutive $S$-set.
Then the product in $X$ defined by $x\cdot y=xf(y)$ is associative,
and it endows $X$ with the structure of $*$-semigroup
in such a way that $f$ becomes an \'etale $*$-homomorphism. The given $S$-action coincides with the canonical
action associated with $f$.
\end{proposition}
\begin{proof}
For associativity observe that
\[
(x\cdot y)\cdot z=(xf(y))\cdot z=xf(y)f(z)=xf(yf(z))=x\cdot (yf(z))=x\cdot(y\cdot z)\,.
\]
It follows that $f$ is a $*$-homomorphism.
The semigroup $X$ is a $*$-semigroup because
\[
x\cdot x^*\cdot x=xf(x^*\cdot x)=xf(x^*f(x))=xf(x)^*f(x)=x\dd(f(x))=x\,.
\]
To see that $f$ is \'etale we fix $x\in X\,$ and
show that $f:x\cdot X\ra f(x)S$ is a bijection.
For surjectivity, since for any $s\in S$, we have $f(xs)=f(x)s\,$,
it suffices to show that $xs\in x\cdot X\,$,
which follows from the following identifications:
\[
xs=xf(x^*)f(x)s=xf(x^*)f(xs)=\cc(x)\cdot(xs)\,.
\]
For injectivity suppose that $f(y)=f(z)\,$ for $y,z\in x\cdot X\,$.
Then
\[
y=\cc(x)\cdot y=\cc(x)f(y)=\cc(x)f(z)=\cc(x)\cdot z=z\,.
\]
It is then clear that the given $S$-action is the canonical one associated with $f\,$.
\end{proof}

\begin{definition}\label{def:strong}
A {\em strongly involutive\/} $S$-set is an involutive $S$-set $f:X\ra S$
such that for the product of Prop. \ref{prop:Sset} we have the identity
\[
(xs)^*= (\dd(x)s)^*x^*
\]
for all $x\in X$ and $s\in S\,$.
\end{definition}

\begin{remark}\label{rem:leftlift}
Spelled out, the preceding definition reads
\[
(xs)^*=(x^*f(x)s)^*f(x)^*=(x^*f(xs))^*f(x^*)\,.
\]
If $f:X\ra S$ is a strongly involutive $S$-set over a left involutive semigroup $S\,$,
then $X$, equipped with the product of Prop.\ \ref{prop:Sset}, is a left involutive semigroup as well.
Indeed, for all $x,y\in X$ we have $(xy)^*=(xf(y))^*=(\dd(x)f(y))^*x^*=(\dd(x)y)^*x^*\,$.
\end{remark}

\begin{proposition}
An \'etale left $*$-homomorphism of $*$-semigroups $f:X\ra S$
with its canonical action {\em (Def.\ \ref{def:action})\/} is an involutive $S$-set.
If $X$ and $S$ are left involutive semigroups, then $f$ is a strongly involutive $S$-set.
\end{proposition}
\begin{proof}
Requirement (1) in Def.\ \ref{def:Sset} is the first part of Def.\ \ref{def:action},
while (2) follows from Rem.\ \ref{rem:preproj} and Lem.\ \ref{lem:xfy}.

For the second assertion, observe (cf.\  also Rem.\ \ref{rem:comparison} below)
that by Lem.\ \ref{lem:xfy}, the identity map $X\ra(X,\cdot)$ is
a bijective left $*$-homomorphism, where $(X,\cdot)$ is the set underlying $X$ equipped with the product
of Prop.\ \ref{prop:Sset}.
Since $X$ is a left involutive semigroup, it follows from Ex.\ \ref{ex:leftstar}
that $(X,\cdot)$ as well is a left involutive semigroup.
Therefore $f:X\ra S$ is a strongly involutive $S$-set.
\end{proof}

The following theorem sums up the results so far obtained.

\begin{svgraybox}
\begin{theorem}\label{thm:Sset}
Let $S$ be a $*$-semigroup.
\begin{enumerate}
\item[(a)]
The category of involutive $S$-sets and $S$-equivariant $*$-morphisms over $S$
is isomorphic to the category of
\'etale $*$-homomorphisms $X\ra S$ and \'etale $*$-homomorphisms over $S\,$.
\item[(b)]
If $S$ is a left involutive semigroup, then the category of strongly involutive $S$-sets
is isomorphic to the category of
\'etale $*$-homomorphisms $X\ra S$ with left involutive semigroup $X\,$.
\end{enumerate}
\end{theorem}
\end{svgraybox}

\begin{remark}\label{rem:comparison}
We may pass from a given \'etale left $*$-homomorphism $f:X\ra S$
of $*$-semigroups to an involutive $S$-set,
and hence to an \'etale $*$-homomorphism.
This amounts to simply changing the product in $X$
so that $f$ becomes a $*$-homomorphism.
Let us denote the new product in $X$ by
\[
x\cdot y=xf(y)\,,
\]
where the action is the canonical action associated with $f\,$.
Lem. \ref{lem:xfy} implies then that
\[
xy=xf(\dd(x)y)=x\cdot (\dd(x)y)
\]
so that the identity map $\iota:X\ra (X,\cdot)$
is a bijective left $*$-homomorphism
\[
\xymatrix{X \ar[dr]^-f_-{\text{\'etale left $*$-homo}} \ar[rr]^\iota
&& (X,\cdot) \ar[dl]_-f^-{\text{\'etale $*$-homo}}\\
& S}
\]
which is \'etale by Ex. \ref{ex:fg}.
The upshot is that an \'etale left $*$-homomorphism
is almost a $*$-homomorphism:
it differs from an \'etale $*$-homomorphism only by a bijective \'etale left $*$-homomorphism.
By Lem. \ref{lem:xfy}, the given $f:X\ra S$ is a $*$-homomorphism
if and only if $\iota$ is a $*$-isomorphism.
\end{remark}

\subsection{Balanced $S$-sets}\label{subsec:rightleft}
\leavevmode\par\vspace{1ex}

There is another way to think of a strongly involutive $S$-set.
This is what we shall call a balanced $S$-set.

\begin{definition}\label{def:bal}
Let $S$ be a $*$-semigroup.
{\em A balanced $S$-set\/} is an involutive $S$-set
$f:X\ra S$ (Def.\ \ref{def:Sset}) also satisfying:
\begin{enumerate}
\item{$(x^*r^*)^*s=((xs)^*r^*)^*\,$;}
\item{$(xf(x^*))^*=xf(x^*)\,$,}
\end{enumerate}
for all $x\in X$ and $r,s\in S\,$.
\end{definition}

\begin{remark}\label{rem:bal}
If we define a \emph{left} action $rx=(x^*r^*)^*$ 
then condition 1 of Def.\ \ref{def:bal} reads $(rx)s=r(xs)\,$ and
condition 2 reads $f(x)x^*=xf(x^*)\,$.
If $S$ is an involutive semigroup, then this left action is associative and $f$ is left equivariant.
\end{remark}

\begin{lemma}\label{lem:four}
Let $S$ be a left involutive semigroup and $f:X\ra S$ be a strongly involutive $S$-set.
Then $f$ is a balanced $S$-set.
\end{lemma}
\begin{proof}
Let $(X,\cdot)$ denote the product of Prop.\ \ref{prop:Sset}.
Then $(X,\cdot)$ is a left involutive semigroup because $f$
is strongly involutive and $S$ is left involutive, cf.\  Rem.\ \ref{rem:leftlift}.
Therefore $xf(x^*)=x\cdot x^*$ is the codomain projection of $x$ in $(X,\cdot)$ and thus self-adjoint,
establishing condition 2 of a balanced $S$-set.
Condition 1 may be rewritten as $(ys)^*=(xs)^*r^*$ for $y=(x^*r^*)^*\,$.
Now we get
\[
(ys)^*=(\dd(y)s)^*y^*=(\cc(x^*r^*)s)^*x^*r^*=(\cc(x^*)s)^*x^*r^*=(\dd(x)s)^*x^*r^*=(xs)^*r^*
\]
as required.
Therefore, $f$ is a balanced $S$-set.
\end{proof}

\begin{proposition}\label{prop:bal}
Let $S$ be a $*$-semigroup.
Then a balanced $S$-set is a strongly involutive $S$-set.
In particular, for a left involutive semigroup $S\,$,
balanced $S$-sets are the same as strongly involutive $S$-sets.
\end{proposition}
\begin{proof}
The second assertion follows from the first together with Lem.\ \ref{lem:four}.
Suppose that an involutive $S$-set $f:X\ra S$ is balanced.
Then $f$ satisfies
\[
f(x)(x^*f(x))=(f(x)x^*)f(x)=xf(x^*)f(x)=x
\]
so that
\[
(xs)^*=s^*x^*=s^*(f(x^*)(xf(x^*)))=(s^*f(x^*))(xf(x^*))
\]
\[
=(f(x)s)^*(xf(x^*))
=(f(xs)^*x)f(x^*)=(x^*f(xs))^*f(x^*)
\]
and $f$ is a strongly involutive $S$-set, cf.\  Rem.\ \ref{rem:leftlift}.
\end{proof}

\section{The classifying topos of a left involutive semigroup}
\vspace{2ex}

Throughout this section $S$ denotes a \emph{left involutive} semigroup.
The main categories associated with $S$ are a category $\cX(S)$ of $*$-semigroups over $S\,$,
a small category $L(S)\,$ of copartial maps between projections of $S$,
and the topos $\B(S)$ of set-valued presheaves on $L(S)\,$.
For an inverse semigroup $S$ the topos $\B(S)$ is known to be a model
for the \emph{classifying topos} of $S\,$.
This is also true if $S$ is just
a left involutive semigroup, cf.\  Rem.\ \ref{rem:LS} below.
The main contribution of this section is the explicit construction of a functor
$\Lambda:\B(S)\ra\X(S)$,
a kind of enhanced \emph{Grothendieck construction}, made possible through the combinatorial description
of left involutive semigroups \'etale over $S$ as strongly involutive $S$-sets, cf.\  \S~\ref{subsec:action}.

\subsection{The categories $\cX(S)$ and $L(S)$}
\leavevmode\par\vspace{1ex}

Let $\cX(S)$ denote the following category associated with a left involutive semigroup $S$:
{\em objects} are $*$-homomorphism $f:X\ra S\,$,
where $X$ is a $*$-semigroup; {\em morphisms} $\psi:f\ra g$ are commutative triangles
\[
\xymatrix{X \ar[dr]_-f \ar[rr]^-\psi && Y \ar[dl]^-g\\
& S}
\]
where $\psi$ is a \emph{left $*$-homomorphism}.

\begin{remark}
The class of morphisms we have chosen for $\cX(S)$ is larger than the class representing
the objects. We must take left $*$-homomorphisms as morphisms and not $*$-homomorphisms,
for the purpose of describing the adjointness $\Lambda\adj\Gamma\,$,
because in general its counit is \emph{not\/} a $*$-homomorphism, so that Prop. \ref{prop:counit}
is valid with the proposed definition of $\cX(S)\,$.
\end{remark}

Let $L(S)$ denote the following category associated with a left involutive semigroup $S\,$:
its \emph{objects} are the projections of $S\,$,
and its \emph{morphisms} $s:d\ra e$ are elements $s\in S$ such that $\dd(s)=s^*s=d$ and $s=es\,$.
By Lem. \ref{lem:leftright} the latter means $\cc(s)\leq e$.
In other words, a morphism of $L(S)$
is a pair $(s,e)\in S\times P(S)$ satisfying $\cc(s)\leq e\,$.
The domain of $(s,e)$ is $\dd(s)$ and its codomain is $e\,$.
Note that the relation $\cc(s)\leq e$ may be viewed as a morphism $\cc(s):\cc(s)\ra e$ in $L(S)$.

\begin{remark}\label{rem:LS}
Composition of $s:d\ra e$ and $t:e\ra f$ in $L(S)$ is defined to be $ts:d\ra f$.
This is consistent because $S$ is corestrictive by Lem. \ref{lem:birestrictive} so that $\cc(ts)\leq\cc(t)$.
The following diagram depicts composition in $L(S)\,$:
\[
\xymatrix{d \ar[rd]_{s} \ar[r]^s &\cc(s) \ar[d]^{\cc(s)} \ar[r]^{t\cc(s)}& \cc(ts) \ar[d]^{\cc(ts)}\\
& e\ar[rd]_{t}\ar[r]^t & \cc(t)\ar[d]^{\cc(t)}\\
&& f}
\]
The horizontal arrows are elements $s\in S$ viewed as morphisms $\dd(s)\ra\cc(s)$ in $L(S)\,$,
the vertical arrows are inclusion relations between projections $e\in P(S)$ also viewed as morphisms in $L(S)\,$,
and the diagonal arrows are general morphisms in $L(S)\,$.
The whole diagram can be viewed as a commutative diagram in $L(S)\,$.

By Lem.\ \ref{lem:reduct}, $S$ is locally involutive,
and thus has an associated ordered groupoid $\sG(S)$ by Prop. \ref{prop:groupoid}.
Inside $\sG(S)$ the morphism $t\cc(s)$ may be viewed as $t$ restricted to $\cc(s)$.
Note that $\cc(s)\leq\dd(t)$ so that the product $ts\in S$ represents the composition
(in $L(S)$ as well in $\sG(S)$)
of $s$ with the restriction of $t$ to $\cc(s)\,$, cf.\  Rem.\ \ref{rem:extended}.
The upper horizontal line above shows that we have $\cc(t\cc(s))=\cc(ts)$ in $S\,$.
We invite the reader to deduce this identity directly from the definition of a left involutive semigroup and the hypothesis $\cc(s)\leq\dd(t)$, cf.\  also Rem.\ \ref{rem:three}.
\end{remark}

\subsection{The classifying topos $\B(S)$}\label{subsec:BS}
\leavevmode\par\vspace{1ex}

The category $L(S)$ is an example of the \emph{left cancellative category} associated with an \emph{ordered groupoid},
cf.\  Lawson \cite{L4}.
In particular, the classifying topos of $\sG(S)$ is equivalent to the topos $\B(S)$ of set-valued presheaves on $L(S)$
so that it is justified to call $\B(S)$ {\em the\/} classifying topos of $S\,$, cf.\  Prop.\ 1.14 and Prop. 2.37 \cite{F7}.
A typical object of $\B(S)$ is thus a functor
\[
P: L(S)^\op\ra \Set\,.
\]
For such a presheaf $P$ and any morphism $s:d\ra e$ in $L(S)$ we have a function
\[
P(s):P(e)\ra P(d)\,,
\]
which we call \emph{transition} in $P$ along $s\,$.
We use $x.s$ to denote $P(s)(x)$ when $x\in P(e)$ so that $x.s\in P(d)\,$.
If $d\leq e\,$, then there is a morphism $d:d\ra e$ in $L(S)$ and (by a slight abuse of notation)
we write $x.d$ for the \emph{restriction\/} of $x\in P(e)$ to $P(d)\,$.
For instance, if $x\in P(e)\,$, then $x.e=x\,$.
This is consistent with the terminology used for the ordered groupoid $\sG(S)$ constructed in Prop. \ref{prop:groupoid}.

\begin{example}\label{ex:rep}
The representable presheaf $\widehat{e}$ associated with a projection $e$
is given by
\[
\widehat{e}(d)=L(S)(d,e)=\{s\mid \dd(s)=d\,;\, \cc(s)\leq e\,\}\,.
\]
The Yoneda embedding provides a full and faithful functor
\begin{equation}\label{eq:Yoneda}
L(S)\ra\B(S)\;:\;e\mapsto\widehat{e}\,.
\end{equation}
\end{example}

\subsection{From $\B(S)$ to $\cX(S)$}
\leavevmode\par\vspace{1ex}

Let $P$ be a presheaf on $L(S)\,$. Define
\[
\Lambda(P) = \{(r,x) \mid r\in S\;\text{and}\;x\in P(\cc(r))\,\}\,,
\]
and
\begin{equation}\label{eq:etale}
f:\Lambda(P)\ra S\;:\;(r,x)\mapsto r\,.
\end{equation}
and
\[
(r,x)^*=(r^*,x.r)\,.
\]
Then $f:\Lambda(P)\ra S$ is an involutive $S$-set,
for the following right $S$-action:
\[
(r,x)s=(rs, x.\cc(rs))\,,
\]
which makes sense because $S$ is a corestrictive semigroup so that $\cc(rs)\leq\cc(r)\,$.
Equivariance and unit conditions
of an involutive $S$-set are straightforward to verify.

\begin{proposition}
The involutive set $\Lambda(P)$ equipped with the product
\[
(s,y)\cdot(r,x)=(s,y)r=(sr,y.\cc(sr))
\]
is a left involutive semigroup, and $f:\Lambda(P)\ra S$ is an \'etale $*$-homomorphism.
\end{proposition}
\begin{proof}
By Rem. \ref{rem:leftlift} it is enough to show that
in $\Lambda(P)$ the identity
\[
((r,x)s)^*=(\dd((r,x))s)^*\cdot (r,x)^*
\]
holds.
Note first that $\dd((r,x))=(\dd(r),x.r)$ and
\[
((\dd(r),(x.r))s)^*=(\dd(r)s,x.r\cc(\dd(r)s))^*
\]
so that the right hand side above equals $(\dd(r)s)^*r^*,x.rs)\,$.
On the other hand, the left hand side equals $((r,x)s)^*=(rs,x.\cc(rs))^*=((rs)^*,x.rs)\,$.
We are done since in any left involutive semigroup $S$ we have $(\dd(r)s)^*r^*=(rs)^*\,$.
\end{proof}

\begin{remark}[Projections in $\Lambda(P)$]\label{rem:projlamP}
The projections of $\Lambda(P)$ are precisely the pairs $(e,x)$
such that $e$ is a projection of $S\,$, cf.\  Lem.\ \ref{lem:reflect}.
Observe that many additional properties of $S$ are inherited by $\Lambda(P)$
via the \'etale $*$-homomorphism $f:\Lambda(P)\ra S\,$.
For instance, whenever projections in $S$ compose (e.g.\ in an inverse semigroup $S$)
then projections in $\Lambda(P)$ compose as well, which has
quite strong implications for $\Lambda(P)\,$: its projections form a left regular band by Lem.\ \ref{lem:birestrictive}.
However, certain properties are not inherited.
For instance, even if $S$ is an inverse semigroup, the semigroup $\Lambda(P)$ may not be inverse.
Two projections $(d,x)$ and $(e,y)$ in $\Lambda(P)$ commute if and only if
\[
(de,x.de)=(d,x)\cdot(e,y)=(e,y)\cdot(d,x)=(ed,y.ed)
\]
if and only the projections $e$ and $d$ commute and the restrictions of $x$ and $y$ coincide.
If $(e,x)$ is a projection of $\Lambda(P)\,$,
then the unique lifting of an equation $t=et$ in $S$ is the pair $(t,x.\cc(t))\,$ in $\Lambda(P)$. Indeed,
\[
(e,x)\cdot(t,x.\cc(t))=(e,x)t=(et,x.\cc(et))=(t,x.\cc(t))\,.
\]
\end{remark}

\begin{remark}[Functoriality]
Our construction yields a functor $\Lambda:\B(S)\ra\cX(S)\,$.
If $\gamma:P\ra Q$ is a natural transformation,
then $\Lambda(\gamma):\Lambda(P)\ra \Lambda(Q)$ is defined by
\[
\Lambda(\gamma)(r,x) = (r, \gamma_{\cc(r)}(x))\,.
\]
This is a $*$-homomorphism: $*$-preservation
$\Lambda(\gamma)((r,x)^*)=\Lambda(\gamma)(r,x)^*$ follows from the
naturality of $\gamma$ for $r:\dd(r)\ra \cc(r)\,$, and product compatibility
\[
\Lambda(\gamma)((s,y)\cdot(r,x))=\Lambda(\gamma)(s,y)\cdot\Lambda(\gamma)(r,x)
\]
follows from the naturality of $\gamma$ for
$\cc(sr)\leq \cc(s)\,$.
\end{remark}

\begin{remark}[Restriction to representable presheaves]\label{rem:rep}
Let us make explicit the restriction $S:L(S)\ra\cX(S)$ of $\Lambda:\B(S)\ra\cX(S)$
along the Yoneda embedding:
\[
\xymatrix{L(S) \ar[d]_-{\eqref{eq:Yoneda}} \ar[drr]^-{S}\\
\B(S) \ar[rr]^-\Lambda && {\cX(S)}}
\]
so that $S(e)=\Lambda(\widehat{e})$ for any $e\in P(S)\,$, cf.\  Ex.\ \ref{ex:rep}.
Unraveling the definitions yields
\[
S(e)=\{(r,s)\in S\times S\mid s\in L(S)(\cc(r),e)\}=\{ (r,s)\in S\times S\mid \dd(s)=\cc(r)\,,\; \cc(s)\leq e\}
\]
with product and involution given by
\[
(p,q)\cdot(r,s)=(pr, q.\cc(pr))\;;\;(r,s)^*=(r^*,sr)\,.
\]
Then $S(e)$ is a left involutive semigroup and $f_e:S(e)\ra S:(r,s)\mapsto r$ is an \'etale $*$-homomorphism.
By definition we get the following pullback square in sets

\begin{equation}\label{eq:SeS}
\xymatrix{S(e) \ar[d]_-{f_e} \ar[r]^-{(r,s)\mapsto s} & eS \ar[d]^-{\dd(-)}\\
S \ar[r]^-{\cc(-)} & P(S)}
\end{equation}
using that $eS=S\downarrow e=\{s\in S\mid \cc(s)\leq e\}$, cf.\  Rem.\ \ref{rem:coset}.
When we identify the projections of $S(e)$ with $eS$ (cf.\  Rem.\ \ref{rem:projlamP})
we have $\dd((r,s))=sr$ and $\cc((r,s))=s\,$.
Thus, the top map in \eqref{eq:SeS} is the codomain map for $S(e)\,$
and the pullback square also expresses the fact that $f_e:S(e)\ra S$ is \'etale, cf.\ Prop. \ref{prop:lift}.
We shall see below that if $S$ is an inverse semigroup,
then $S(e)$ is an inverse semigroup if and only if the projections of $S(e)$ commute
if and only if $eS$ is a left compatible subset of $S$ (Rem.\ \ref{rem:Seinv})
if and only if $f_e$ is injective, cf.\ Prop. \ref{prop:inv}.
\end{remark}

\begin{lemma}\label{lem:xirho}
Let $f:X\ra S$ be an \'etale $*$-homomorphism from a $*$-semigroup $X$ to
a left involutive semigroup $S$, and let $e$ be a projection of $S\,$.
If two left $*$-homomorphisms $S(e)\ra X$ over $S$ agree on $(e,e)\,$,
then they agree everywhere.
\end{lemma}
\begin{proof}
Let $\xi,\rho$ be two such left $*$-homomorphisms,
which are necessarily $*$-homomorphisms by Prop. \ref{prop:starhomo}, and assume that $s=es$ in $S\,$.
Use the equation
\[
(s,\cc(s))=(e,e)\cdot(s,\cc(s))
\]
in $S(e)$ and the unique lifting property of $f:X\ra S$
to conclude that $\xi$ and $\rho$ agree on $(s,\cc(s))\,$.
Therefore, $\xi$ and $\rho$ also agree on the domain
$\dd((s,\cc(s)))=(\dd(s),s)$
in $S(e)\,$.
Let $(r,s)$ be an arbitrary element of $S(e)\,$,
so that $\cc(r)=\dd(s)$ and $\cc(s)\leq e$.
We have $f_e(\xi(r,s))=r=f_e(\rho(r,s))\,$.
Use the equation
\[
(r,s)=(\dd(s),s)\cdot(r,s)
\]
and the unique lifting property of $f:X\ra S$ to conclude
that $\xi=\rho\,$.
\end{proof}

\section{The adjointness $\Lambda\adj\Gamma$}
\vspace{2ex}

Throughout this section $S$ denotes a left involutive semigroup.

\subsection{From $\cX(S)$ to $\B(S)$}
\leavevmode\par\vspace{1ex}

Let $f:X\ra S$ be a $*$-homomorphism with $*$-semigroup $X$, viewed
as an object of $\cX(S)$. We probe such an $f$
with left $*$-homomorphisms $f_e:S(e)\ra X$ over $S\,$.
Indeed, such an $f$ engenders a presheaf
\[
\Gamma(f):L(S)^\op\ra \Set
\]
such that $\Gamma(f)(e)=\cX(S)(f_e,f)\,.$
Any $\alpha\in\Gamma(f)(e)$ satisfies $f\alpha=f_e\,$.
Transition in $\Gamma(f)$ along a morphism $s:d\ra e$ of $L(S)$ is given by precomposition with $S(s)$:
\[
(\alpha.s)(p,q)=\alpha(S(s)(p,q))=\alpha(p,sq)\,,
\]
where $(p,q)\in S(d)\,$.
A routine verification shows that $\Gamma:\cX(S)\ra \B(S)$ is a functor.
The aim of this section is to show that the functor $\Lambda$
is left adjoint to $\Gamma\,$,
thereby inducing an equivalence between the classifying topos $\B(S)$
and the category of left-involutive semigroups \'etale over $S\,$.

One way to establish an adjointness $\Lambda\adj\Gamma$ is by showing that
there is a natural bijection between natural transformations
\[
P\ra \Gamma(f)
\]
of presheaves on $L(S)\,$, and left $*$-homomorphisms over $S$
\[
\xymatrix{
\Lambda(P) \ar[dr] \ar[rr] && X \ar[dl]^-f\\
& S}\,.
\]
Another way is by exhibiting a unit natural transformation $\eta:\id_{\B(S)}\ra\Gamma\Lambda$
and a counit natural transformation $\epsilon:\Lambda\Gamma\ra\id_{\cX(S)}$ satisfying so-called
triangle identities.
We shall follow the second method.

\subsection{Unit and counit}
\leavevmode\par\vspace{1ex}

The unit associated with a presheaf $P$ is given by
\begin{equation}\label{eq:unitd}
\eta(P)_d:P(d)\ra \{\, \text{left $*$-homomorphisms}\; S(d)\ra \Lambda(P)\;\text{over}\;S\,\}\,,
\end{equation}
such that for $a\in P(d)$
\begin{equation}\label{eq:alphax}
\eta(P)_d(a)(r,s)=(r,a.s)\,;\;(r,s)\in S(d)\,.
\end{equation}
Note that $a.s\in P(\cc(r))$ because $\dd(s)=\cc(r)\,$.
We sometimes denote $\eta(P)_d(a)$ more simply by $\eta_d(a)\,$.
It is immediately verified that
$\eta_d(a)$ is a $*$-homomorphism.
Also note that by Prop.\ \ref{prop:starhomo} a left $*$-homomorphism $S(d)\ra \Lambda(P)$ over $S$
is necessarily a $*$-homomorphism
since $\Lambda(P)\ra S$ is an \'etale $*$-homomorphism.

We verify immediately that $\eta(P)$ is a natural transformation,
and that $\eta$ is natural, meaning that
the naturality square
\[
\xymatrix{
P \ar[rr]^-{\eta(P)} \ar[d]_-\gamma && \Gamma\Lambda(P) \ar[d]^-{\Gamma\Lambda(\gamma)}\\
Q \ar[rr]^-{\eta(Q)} && \Gamma\Lambda(Q)}
\]
commutes.

\begin{proposition}\label{prop:unitiso}
For any presheaf $P\,$, the unit $\eta(P)$ is an isomorphism.
\end{proposition}
\begin{proof}
We have $\eta_d(a)(d,d)=(d,a.d)=(d,a)$ so that \eqref{eq:unitd} is injective.
Now suppose that $\alpha:S(d)\ra \Lambda(P)$ is a $*$-homomorphism over $S\,$.
We have $\alpha(d,d)=(d,a)$ for some $a\in P(d)$ because $\alpha(d,d)$
must be a projection of $\Lambda(P)\,$.
We have
\[
\alpha(d,d)=(d,a)=\eta_d(a)(d,d)\,.
\]
By Lem.\ \ref{lem:xirho}, we have $\alpha=\eta_d(a)\,$,
so that \eqref{eq:unitd} is surjective, whence an isomorphism.
\end{proof}

\begin{remark}
For instance, for any projection $e$ of $S$, the unit
\[
\eta(\widehat{e}):\widehat{e}\ra \Gamma\Lambda(\widehat{e})=\Gamma(S(e))
\]
is an isomorphism.
Thus, for any other projection $d$, the map induced by $S$
\[
L(S)(d,e)\ra \{\,\text{$*$-homomorphisms}\; S(d)\ra S(e)\;\text{over}\; S\,\},
\]
is a bijection. Therefore, the functor $S$ is full and faithful,
and the following diagram
\[
\xymatrix{L(S) \ar[d]_-{\eqref{eq:Yoneda}} \ar[drr]^-{S}\\
\B(S) && \ar[ll]_-\Gamma {\cX(S)}}
\]
commutes.
\end{remark}

We now turn to the \emph{counit}
\[
\xymatrix{
\Lambda(\Gamma(f)) \ar[dr] \ar[rr]^{\varepsilon_f} && X \ar[dl]^-f\\
& S}
\]
associated with an object $*$-homomorphism $f:X\ra S$ of $\cX(S)\,$.
We define
\[
\varepsilon_f(r,\xi)=\xi_{\cc(r)}(r,\cc(r))\,,
\]
where $\xi\in\Gamma(f)$, $(r,\cc(r))\in S(\cc(r))$ and $\xi_{\cc(r)}:S(\cc(r))\ra X$
is the left $*$-homomorphism over $S$ represented by $(r,\xi_{\cc(r)})\in\Lambda(\Gamma(f))\,$.

\begin{lemma}\label{lem:rsrs}
For any elements $r,s$ of a left involutive semigroup $S$
the following diagram  commutes in $L(S)\,$.
\[
\xymatrix{
\dd(rs)=\dd(\dd(r)s)\ar[r]^-{\dd(r)s}\ar[rd]_{\dd(r)s} \ar@/^2pc/[rr]^{rs}
& \cc(\dd(r)s) \ar[d] \ar[r]_{r\cc(\dd(r)s)} & \cc(rs)\ar[d]\\
& \dd(r)\ar[r]_r&\cc(r)}
\]
\end{lemma}
\begin{proof}
We have $\dd(rs)=(rs)^*(rs)=(\dd(r)s)^*r^*rs=(r^*rs)^*(r^*rs)=\dd(\dd(r)s)\,$.
\end{proof}

\begin{proposition}\label{prop:counit}
For each object $f:X\ra S$ of $\cX(S)\,$,
the counit $\varepsilon_f$ belongs to $\cX(S)\,$.
The maps $\varepsilon_f$ induce a natural transformation $\varepsilon:\Lambda\Gamma\ra\id_{\cX(S)}\,$.
\end{proposition}
\begin{proof}
We have to show that $\varepsilon_f:\Lambda(\Gamma(f))\ra X$ is a left $*$-homomorphism over $S$
knowing that $\xi_{\cc(r)}:S(\cc(r))\ra X$ is a left $*$-homomorphism
for each $(r,\xi_{\cc(r)})\in\Lambda(\Gamma(f))\,$.
We shall abbreviate $\xi_{\cc(r)}=\xi$ and $\xi_{\cc(s)}=\eta\,$.
Applying the various definitions we get
\begin{enumerate}
\item{$\varepsilon_f(r,\xi)^*=\xi(r,\cc(r))^*=\xi((r,\cc(r))^*)=\xi(r^*,r)\,$;}
\item{$\varepsilon_f((r,\xi)^*)=\varepsilon_f(r^*,\xi.r)=(\xi.r)(r^*,\dd(r))=\xi(r^*,r\dd(r))=\xi(r^*,r)\,$;}
\item{$\varepsilon_f((r,\xi)(s,\eta))=\varepsilon_f(rs,\xi.\cc(rs))=\xi(rs,\cc(rs)\cc(rs))=\xi(rs,\cc(rs))\,$;}
\item{$\varepsilon_f(r,\xi)\varepsilon_f(\dd(r,\xi)(s,\eta))=\varepsilon_f(r,\xi)\varepsilon_f((\dd(r),\xi.r)(s,\eta))\\
    =\varepsilon_f(r,\xi)\varepsilon_f(\dd(r)s,\xi.r\cc(d(r)s))\\
    =\xi(r,\cc(r))\xi(\dd(r)s,r\cc(\dd(r)s))=\xi(rs,\cc(rs))\,$,}
\end{enumerate}
where the last identification follows from Lem.\ \ref{lem:rsrs} and the fact that $\xi$ is a presheaf on $L(S)\,$.
Equating 1 and 2 as well as 3 and 4 shows that $\varepsilon_f$ is a left $*$-homomorphism as asserted.
A routine calculation verifies that $\varepsilon$ is natural transformation.
\end{proof}

Let us verify the triangle identity
\begin{equation}\label{eq:tri}
\xymatrix{
\Lambda(P) \ar@{=}[drr] \ar[d]_-{\Lambda(\eta(P))}\\
\Lambda\Gamma(\Lambda(P)) \ar[rr]^-{\varepsilon_{\Lambda P}} && \Lambda(P)}
\end{equation}
for a presheaf $P\,$.
Let $(r,a)\in \Lambda(P)$ so that $a\in P(\cc(r))\,$.
Then
\[
\Lambda(\eta(P))(r,a)=(r,\eta_d(a))\,,
\]
where $\eta_d(a):S(\cc(r))\ra \Lambda(P)$ is defined in \eqref{eq:unitd} and \eqref{eq:alphax}:
we have $\eta_d(a)(p,q)=(p,a.q)\,$.
Then
\[
\varepsilon_{\Lambda P}(r,\eta_d(a))=\eta_d(a)(r,\cc(r))=(r,a.\cc(r))=(r,a)\,.
\]

The other triangle identity
\[
\xymatrix{
\Gamma(f) \ar@{=}[drr] \ar[d]_-{\eta(\Gamma(f))}\\
\Gamma(\Lambda(\Gamma(f))) \ar[rr]^-{\Gamma(\varepsilon_{f})} && \Gamma(f)}
\;\;\;
\xymatrix{
\Gamma(f)(e) \ar@{=}[drr] \ar[d]_-{\eta(\Gamma(f))_e}\\
\Gamma(\Lambda(\Gamma(f)))(e) \ar[rr]^-{\Gamma(\varepsilon_{f})_e} && \Gamma(f)(e)}
\]
is also routinely verified.
In fact, let $s\in S$ be a projection, and $\xi:S(\cc(s))\ra X$ a left $*$-homomorphism over $S\,$.
Then the commutativity of the right triangle above comes down to
the commutativity of
\[
\xymatrix{
S(\cc(s))\ar[drr]^-\xi \ar[d]_-{\eta_{\cc(s)}(\xi)}\\
\Lambda(\Gamma(f)) \ar[rr]^{\varepsilon_{f}} && X}
\]
where $\eta_{\cc(s)}(\xi)(r,s)=(r,\xi.s)\,$: since $\cc(r)=\dd(s)$ we have
\[
\varepsilon_f(\eta_{\cc(s)}(\xi)(r,s))=\varepsilon_f(r,\xi.s)=(\xi.s)(r,\cc(r))
=\xi(r,s\cc(r))=\xi(r,s\dd(s))=\xi(r,s)\,.
\]

\subsection{Fiber presheaf and equivalence}
\leavevmode\par\vspace{1ex}

Like any adjointness with invertible unit the adjointness $\Lambda\dashv\Gamma$ induces
an equivalence of categories between $\B(S)$ and the full subcategory of $\cX(S)$ on those $f:X\ra S$
for which $\varepsilon_f$ is invertible.
The remainder of this section serves to show that the latter consists of
precisely the \emph{\'etale} $*$-homomorphisms $f:X\ra S$ between left involutive semigroups.

We start with giving an alternative, simpler description of the presheaf $\Lambda(f)$
for an \'etale $*$-homorphism $f\,$,
which we refer to as \emph{fiber presheaf\/} of $f$ and denote $P_f$ for distinctiveness,
although it turns out that $P_f$ and $\Lambda(f)$ are $*$-isomorphic, cf.\ Rem.\ \ref{rem:Pf}.

Let $e\in S$ be a projection and let
\[
P_f(e) = f^{-1}(e)\,.
\]
By Lem.\ \ref{lem:reflect} $f$ reflects (and preserves) projections
so that the fiber $P_f(e)$ consists entirely of projections. Transition in $P_f$ is
determined by the unique lifting property of an \'etale $*$-homomorphism,
cf.\ Prop.\ \ref{prop:lift}:
if $(s,e):d\ra e$ is a morphism
of $L(S)\,$, then for $u\in P_f(e)$ the unique lifting of $s=f(u)s$ in $S$
produces a unique $x\in X$ such that $x=ux$, and transition in $P_f$ is defined as
\[
u.s=\dd(x)\,.
\]
In a left involutive semigroup $\dd(x)$ is a projection, and we have
\[
f(u.s)=f(\dd(x))=\dd(f(x))=\dd(s)=d\,.
\]
This defines a presheaf on $L(S)$ because composition in $L(S)$ is defined by
$(s,e)(t,f)=(st,f)$, cf.\ Rem.\ \ref{rem:LS}.
In particular, transition in $P_f$
along an inclusion of projections $d\leq e$ is
the special case of transition along $d:d\ra e$ in $L(S)\,$.
We thus have a left involutive semigroup $\Lambda(P_f)$ whose elements
are pairs $(r,u)\in S\times P(X)$ such that $f(u)=\cc(r)\,$.
Involution and product in $\Lambda(P_f)$ are given by
\[
(r,u)^*=(r^*,u. r)\text{ and }(r,u)\cdot(s,v)=(rs,u.\cc(rs))\,.
\]

\begin{lemma}\label{lem:m}
Each \'etale $*$-homomorphism  $f:X\ra S$ between left involutive semigroups
induces a $*$-isomorphism $m:\Lambda(P_f)\ra X\,$
where $m(r,u)$ is the unique element of $X$ such that $f(m(r,u))=r$ and $\cc(m(r,u))=u\,$.
\end{lemma}
\begin{proof}
By Prop.\ \ref{prop:lift} the map $m$
is a bijection with inverse $m(r,u)\mapsto (r,u)\,$.
Since $(r,u)^*=(r^*,u.r)$ and $f(m(r,u)^*)=r^*$
and $\dd(u.r)=\cc(r^*)$ we have $m(r,u)^*=m((r,u)^*)\,$.
By Thm.\ \ref{thm:Sset} we get a $*$-isomorphism over $S\,$.
\end{proof}

\begin{proposition}\label{prop:varf}
A $*$-homomorphism $f:X\ra S$ from a $*$-semigroup $X$ to a left involutive semigroup $S$
has an invertible counit $\varepsilon_f$ if and only if $f$ is \'etale and $X$ is left involutive.
\end{proposition}
\begin{proof}
If the counit $\varepsilon:\Lambda\Gamma(f)\ra X$ is invertible,
then $X$ is left involutive by Ex.\ \ref{ex:leftstar}.
The inverse of the counit is also a bijective left $*$-homomorphism, and as such \'etale.
Therefore, $f:X\ra S$ is the composite of two \'etale left $*$-homomorphisms,
whence $f$ is an \'etale $*$-homomorphism by Prop. \ref{prop:starhomo}.

Conversely, if $f$ is an \'etale $*$-homomorphism between left involutive semigroups,
then $X$ is $*$-isomorphic to $\Lambda(P_f)$ by Lem.\ \ref{lem:m}.
By the triangle identity, the counit $\varepsilon_{\Lambda(P_f)}$ is then invertible.
The naturality of the counit implies that $\varepsilon_f$ is invertible as well.
\end{proof}

We thus recover by different means Thm.\ 2.44 of \cite{F7}.
In fact, Thm.\ \ref{thm:BSleft} below is an improvement on \cite{F7} in two respects:
it works over a general left involutive semigroup $S$ and it places the equivalence within
the broader context of the adjointness $\Lambda\adj \Gamma\,$.

\begin{svgraybox}
\begin{theorem}\label{thm:BSleft}
Let $S$ be a left involutive semigroup.
Then $\Lambda\adj \Gamma$ restricts to an equivalence
between the classifying topos $\B(S)$ and the full subcategory of $\cX(S)$
on \'etale $*$-homomorphisms $X\ra S$ between left involutive semigroups.
A left $*$-homomorphism between two such objects is an \'etale $*$-homomorphism.
\end{theorem}
\end{svgraybox}

\begin{corollary}\label{cor:BSleft}
Let $S$ be a left involutive semigroup.
Then the classifying topos $\B(S)$ and the category of left involutive $S$-sets
are equivalent.
\end{corollary}
\begin{proof}
This follows from Thm.\ \ref{thm:BSleft} and Thm.\ \ref{thm:Sset}.
\end{proof}

\begin{corollary}\label{cor:bal}
Let $S$ be a left involutive semigroup.
Then the classifying topos $\B(S)$ and the category
of balanced $S$-sets are equivalent.
\end{corollary}
\begin{proof}
This follows from Cor.\ \ref{cor:BSleft} and Prop.\ \ref{prop:bal}.
\end{proof}

\begin{remark}
Kudryavtseva and Lawson \cite{KL} have also studied this aspect of $\B(S)$ establishing
for an inverse semigroup $S$ and by different means a result  (therein Thm.\ 1.1)
similar to Thm.\ \ref{thm:BSleft} above and its corollaries.
\end{remark}

\begin{remark}\label{rem:Pf}
The \emph{fiber presheaf} $P_f$ associated with an \'etale $f$
coincides with the \emph{sections presheaf} $\Gamma(f)\,$, though $\Gamma(f)$
is defined for any $*$-homomorphism $f\,$,
not just an \'etale one.
In fact, for \'etale $f$ the transpose of the $*$-isomorphism $m:\Lambda(P_f)\ra X$ of Lem.\ \ref{lem:m}
\[
\widehat{m}=\Gamma(m)\cdot \eta(P_f):P_f\ra \Gamma(f)
\]
is an isomorphism of presheaves since $\eta(P_f)$ is (Prop.\ \ref{prop:unitiso}).
Its inverse is given by
\begin{equation}\label{eq:fPf}
\Gamma(f)(e)\ra P_f(e):\alpha\mapsto\alpha(e,e)
\end{equation}
where the (left) $*$-homomorphism $\alpha:S(e)\ra X$ is an element of $\Gamma(f)(e)\,$.
Lem.\ \ref{lem:xirho} says that $\alpha(e,e)$ completely determines $\alpha\,$.
Any element of the fibre $f^{-1}(e)$ is represented by such an $\alpha:S(e)\ra X\,$.
\end{remark}

\section{Involutive semigroups over an inverse semigroup}
\vspace{2ex}

Throughout, $S$ denotes an inverse semigroup
with subset of idempotents $E(S)=E\,$.
A left involutive semigroup is involutive if and only if $(uy)^*=y^*u$ for every projection $u$
and element $y\,$, cf.\  Prop.\ \ref{prop:commproj}.

\subsection{Involutive $\Lambda(P)$}
\leavevmode\par\vspace{1ex}

We can put the adjointness $\Lambda\adj \Gamma$ to work by using it
to identify those presheaves on $L(S)$  that correspond to \'etale $*$-homomorphisms $X\ra S$
for which $X$ is involutive.
In other words, we may use the adjointness to identify those presheaves $P$ for which $\Lambda(P)$
is involutive.
Equivalently, we may
identify $\Gamma(f)$ (or $P_f$ by Rem.\ \ref{rem:Pf})
for those \'etale $*$-homomorphisms $f:X\ra S$ for
which $X$ is involutive.

We remind the reader that we denote transition in a presheaf $P$
along a morphism $s:d\ra e$ by $x.s\,$, $x\in P(e)\,$.

\begin{proposition}\label{prop:inv}
For any presheaf $P$ on $L(S)\,$, the following are equivalent:
\begin{enumerate}
\item{projections commute in $\Lambda(P)\,$;}
\item{$\Lambda(P)$ is inverse;}
\item{$\Lambda(P)$ is involutive;}
\item{$P$ is a subobject of $1$ in the topos $\B(S)\,$;}
\item{$\Lambda(P)\ra S$ is injective.}
\end{enumerate}
\end{proposition}
\begin{proof}
Let $P$ be a presheaf.

\noindent $1 \Rightarrow 2\,$. This holds by Prop.\ \ref{prop:commproj}.

\noindent $2\Rightarrow 3\,$.  This is well known.

\noindent $3\Rightarrow 4\,$.
Let $(e,a)$ be a projection of $\Lambda(P)\,$, so that $e$ is idempotent and $a\in P(e)\,$.
We have $(e,a)^*=(e,a)\,$.
For any other element $(r,b)$ we have
\begin{equation}\label{eq:aer}
((e,a)\cdot(r,b))^*=(er,a.\cc(er))^*=(r^*e,a.\cc(er)er)=(r^*e, a.er)\,,
\end{equation}
and
\begin{equation}\label{eq:ber}
(r,b)^*\cdot(e,a)^*=(r^*,b.r)\cdot(e,a)=(r^*e,(b.r).r^*e(r^*e)^*)=(r^*e,b.er)\,,
\end{equation}
where
\[
\xymatrix{
r^*er \ar[d]_-{er} \ar[r]^-{er} & e \\
rr^*}
\]
are generally different morphisms of $L(S)\,$.
Then \eqref{eq:aer} and \eqref{eq:ber} are equal just when
\begin{equation}\label{eq:aerr}
a.er = b.er\,.
\end{equation}
If \eqref{eq:aerr} holds for every $a\in P(e)$ and $b\in P(rr^*)\,$,
then taking $r=e$ we have
\[
a=a.e=b.e=b
\]
for every $a,b\in P(e)\,$.
Thus, $P(e)$ is either empty or consists of a single element,
so that $P$ is a subobject of $1$ in $\B(S)\,$.

\noindent $4\Rightarrow 5\,$.
If $P$ is a subobject of $1\,$,
then clearly $f:\Lambda(P)\ra S\,$, $f(r,u)=r\,$, is injective.

\noindent $5\Rightarrow 1\,$. This holds because projections commute in $S\,$.
\end{proof}

\begin{corollary}
For any left involutive semigroup $X\,$, and \'etale $*$-homomorphism $f:X\ra S\,$,
the following are equivalent:
\begin{enumerate}
\item{projections commute in $X\,$;}
\item{$X$ is inverse;}
\item{$X$ is involutive;}
\item{$\Gamma(f)\iso P_f$ is a subobject of $1\,$;}
\item{$f$ is injective.}
\end{enumerate}
\end{corollary}
\begin{proof}
Prop.\ \ref{prop:inv} and Thm.\ \ref{thm:BSleft}.
\end{proof}

\begin{corollary}\label{cor:Se}
For any idempotent $e\,$,
$S(e)$ is inverse if and only if $S(e)$ is involutive if and only if
the representable presheaf $\widehat{e}$ is a subobject of $1\,$.
\end{corollary}
\begin{proof}
This holds because $S(e)=\Lambda(\widehat{e})\,$.
\end{proof}

\begin{remark}
Thus, the collection of
\'etale $*$-homomorphisms $X\ra S$ for which $X$ is involutive identifies the localic reflection of $\B(S)$
for this is precisely the frame of subobjects of $1\,$.
A sheaf on this locale may be identified
as a quotient of a coproduct of subobjects of $1\,$.
For instance, the underlying set of the coproduct
\[
f:X\oplus Y\ra S
\]
of two \'etale $*$-homomorphisms $X\ra S$ and $Y\ra S\,$,
where both $X$ and $Y$ are involutive, is the disjoint union of sets $X$ and $Y\,$.
The product of $u,v\in X\oplus Y$ is given by the unique element $uv$ such
that $uv=uu^*(uv)$ and $f(uv)=f(u)f(v)\,$.
For instance, if $u\in X\,$, then $uv\in X\,$.
The left involutive semigroup $X\oplus Y$ is involutive
if and only if $X$ and $Y$ are disjoint in the sense that the
commutative square
\[
\xymatrix{
0 \ar[r] \ar[d] & Y \ar[d]\\
X \ar[r] & S}
\]
is a pullback.
\end{remark}

\begin{remark}
The following concepts coincide.
In each case, the collection of all of them forms a frame,
identifying the localic reflection of $\B(S)\,$.
\begin{enumerate}
\item{An injective \'etale $*$-homomorphism $f:X\ra S$ for which $X$ is a left involutive semigroup.}
\item{An \'etale $*$-homomorphism $f:X\ra S$ for which $X$ is involutive.}
\item{An \'etale $*$-homomorphism $f:X\ra S$ for which $X$ is inverse.}
\item{An inverse subsemigroup $T\subseteq S$ that is also
a right ideal in the sense that
\[
\forall s,t\in S\,:\;t\in T\Rightarrow ts\in T\,.
\]
In fact, the \'etale property implies that the right ideal property holds for $t=d$ idempotent,
and then for general $t\in T$ conclude first that $t^*t\in T$ (since $T$ is an inverse subsemigroup),
whence $t^*ts\in T$ for any $s\,$, so that $ts=tt^*ts\in T$ again by closure under multiplication.}
\item{A right ideal $T\subseteq S$ that is closed under $*\,$.}
\item{A two-sided ideal $T\subseteq S\,$.
Since $t^*=t^*tt^*$ we see that a two-sided ideal is closed under $*\,$.}
\item{A `normal' subset $D\subseteq E$ of idempotents,
meaning one that is closed under conjugation:
\[
\forall s,d\,:\;d\in D\Rightarrow s^*ds\in D\,.
\]
If a subset $D\subseteq E$ is normal,
then $T=\{ s\mid s^*s\in D\,\}\subseteq S$ is a two-sided ideal:
for instance, if $s^*s\in D\,$, then $ss^*=s(s^*s)s^*\in D\,$,
showing that $T$ is closed under $*\,$.
If $T\subseteq S$ is a two-sided ideal,
then $D=E\cap T$ is a normal subset of $E\,$.}
\end{enumerate}
\end{remark}

\subsection{Left compatibility}
\leavevmode\par\vspace{1ex}

Let $S$ denote an inverse semigroup with idempotent poset $E\,$.
Consider the domain restriction $se$ of an element $s$ to an idempotent $e\,$.
The domain of $se$ is the infimum $s^*s\wedge e=s^*se\,$.
Two elements $s,t$ of $S$ are said to be {\em left compatible\/} \cite{L} if $st^*t=ts^*s\,$:
this says that restrictions of $s$ and $t$ to $s^*s\wedge t^*t$ are equal,
in which case the common restriction is the infimum $s\wedge t$ in the usual
ordering of $S\,$.
It follows that two elements $s,t$ are left compatible if and only if $st^*$
is an idempotent, in which case $st^*$ is the codomain of the common restriction $s\wedge t\,$.

\begin{remark}\label{rem:Seinv}
For brevity let $(p,q)(r,s)$ denote the product in the representable left involutive semigroup $S(e)$
(Rem.\ \ref{rem:rep}).
Observe that two projections $(s^*s,s)$
and $(r^*r,r)$ of $S(e)$ commute if and only if $r,s$ are left compatible elements of $S\,$.
In fact, this is verified by a straightforward calculation: we have
\[
(s^*s,s)(r^*r,r)=(s^*sr^*r,ss^*sr^*r)=(s^*sr^*r,sr^*r)
\]
and likewise
\[
(r^*r,r)(s^*s,s)=(r^*rs^*s,rs^*s)\,.
\]
These two are equal just when $sr^*r=rs^*s\,$, i.e.\ when $r$ and $s$ are left compatible.
Thus, $S(e)$ is inverse if and only if $S(e)$ is involutive if and only if
$eS$ is left compatible in the sense that every pair of elements
$r,s\in eS$ are left compatible.
\end{remark}

When does $(xy)^*=y^*x^*$ hold for two
elements $x,y$ of a representable left involutive semigroup
$S(e)\ra S\,$?
The answer is again left compatibility.

\begin{proposition}\label{prop:sym}
In $S(e)$ we have $((p,q)(r,s))^*=(r,s)^*(p,q)^*$
if and only if $s$ and $qp$ are left compatible in $S\,$.
\end{proposition}
\begin{proof}
This follows from a straight-forward calculation:
\[
((p,q)(r,s))^*=(pr, qprr^*p^*)^*=((pr)^*, qprr^*p^*pr)
\]
\[
=(r^*p^*,qpp^*prr^*r)=(r^*p^*,qpr)
\]
and
\[
(r,s)^*(p,q)^*=(r^*,sr)(p^*,qp)=(r^*p^*,srr^*p^*pr)
\]
\[
=(r^*p^*,sp^*prr^*r)=(r^*p^*,sp^*pr)\,.
\]
Thus, $((p,q)(r,s))^*=(r,s)^*(p,q)^*$ if and only if $qpr=sp^*pr$
if and only if $qprr^*=sp^*prr^*$ if and only if
\begin{equation}\label{eq:qp}
qps^*s=sp^*ps^*s=ss^*sp^*p=sp^*p=s(qp)^*qp
\end{equation}
because $rr^*=s^*s\,$.
Also note that $p^*p=(qp)^*qp$ because $pp^*=q^*q\,$.
Equation \eqref{eq:qp} says precisely that $s$ and $qp$ are left compatible.
\end{proof}

\begin{remark}
Prop.\ \ref{prop:sym} says in other words that in $S(e)$ we have
\[
((p,q)(r,s))^*=(r,s)^*(p,q)^*
\]
if and only if the domain projection of $(p,q)\,$, which is $qp\,$, and the codomain of $(r,s)\,$,
which is $s\,$, are left compatible.
\end{remark}

\section{Involutive $S$-modules and $S$-algebras}\label{sec:inv}
\vspace{2ex}

Throughout, $S$ denotes an inverse semigroup.
We organize the theory around three concepts:
involutive $S$-set (already introduced in Def.\ \ref{def:Sset}),
involutive $S$-module (Def.\ \ref{def:invmod}),
and involutive $S$-algebra (Def.\ \ref{def:Salg}).
The balanced case is also important for all three concepts.
In the module context we shall use the balanced terminology instead of
the equivalent strong terminology (Prop.\ \ref{prop:bal}).

\subsection{Involutive $S$-modules}
\leavevmode\par\vspace{1ex}

\begin{definition}\label{def:invmod}
{\em An involutive $S$-module} is an involutive $S$-set $\psi:\cA\ra S$
(with action denoted $as$) such that:
\begin{enumerate}
\item{$\psi$ has a section $\zero:S\ra \cA\,$,
which is an equivariant $*$-morphism: $\psi(\zero(r))=r\,$, $\zero(s^*)=\zero(s)^*$ and $\zero(rs)=\zero(r)s\,$.}
\item{$\psi$ has a fiberwise commutative addition
\[
\cA\times_S\cA\ra \cA\;;\;(a,b)\mapsto a+b\,,
\]
over $S$ (so $\psi(a+b)=\psi(a)=\psi(b)$);}
\item{$a+\zero(\psi(a))=a$  ($0$-law);}
\item
$(a+b)r=ar+br$  (implicitly $\psi(a)=\psi(b)$ so $\psi(ar)=\psi(a)r=\psi(b)r=\psi(br)\,$);
\item{$(a+b)^*=a^*+b^*\,$.}
\end{enumerate}
{\em A balanced $S$-module\/} is a balanced $S$-set $\psi$ (Def.\ \ref{def:bal}) satisfying the above.
\end{definition}

\begin{remark}\label{rem:zero}
In particular, each fiber $\psi^{-1}(r)$ of an involutive $S$-module $\psi$ is a commutative monoid,
written additively with zero $\zero(r)\,$.
For instance, taking $a=\zero(r)$ in the $0$-law we have
$\zero(r)+\zero(r)=\zero(r)+\zero(\psi(\zero(r))=\zero(r)\,$.
\end{remark}

\begin{remark}\label{rem:bimod}
As in Rem.\ \ref{rem:bal} we may put the definition of a balanced $S$-module equivalently
as an involutive $S$-bimodule in the sense that
$\cA$ has left and right associative actions by $S$ that distribute over addition,
such that $(ra)s=r(as)$ and  $(as)^*=s^*a^*\,$,
and such that $\psi$ and $\zero$ are bi-equivariant.
In fact, if $\psi:\cA\ra S$ is an involutive $S$-module,
then we define a left action by
\[
ra=(a^*r^*)^*\,.
\]
Then Def.\ \ref{def:bal} 1 reads $(ra)s=r(as)\,$,
and $\psi$ is an involutive $S$-bimodule.
We often work with balanced $S$-modules this way.
\end{remark}

\begin{remark}
By  Prop.\ \ref{prop:bal}, when $S$ is inverse,
an involutive $S$-module is balanced if and only if
\[
(as)^*=(a^*\psi(as))^*\psi(a^*)
\]
holds.
This is the strong identity introduced in Def.\ \ref{def:strong}
and Rem.\ \ref{rem:leftlift}.
\end{remark}

\begin{definition}\label{def:Salg}
{\em An involutive $S$-algebra\/}
is an involutive $S$-module $\psi:\cA\ra S$ such that:
\begin{enumerate}
\item{$\cA$ is an involutive semigroup, written multiplicatively ($(ab)^*=b^*a^*$ and $aa^*a=a$);}
\item{$\psi$ and $\zero$ are (multiplicative) homomorphisms,
so that they are $*$-homomorphisms (because by definition $\psi$ and $\zero$ are $*$-morphisms);}
\item{the product distributes over addition:
$(a+b)c=ac+bc\,$, which makes sense because if $\psi(a)=\psi(b)\,$,
then $\psi(ac)=\psi(bc)\,$;}
\item{$(ab)r=a(br)\,$;}
\item{$a(br)^*=(ar^*)b^*\,$.}
\end{enumerate}
{\em A balanced $S$-algebra\/} is a balanced $S$-module satisfying the above.
\end{definition}

\begin{remark}
The other distributive law $a(b+c)=ab+ac$ automatically holds
because of the involution.
Also note that condition 5 is equivalent to $(ar)b=a(rb)\,$, for $rb=(b^*r^*)^*\,$.
\end{remark}

\begin{remark}\label{rem:canon}
A balanced $S$-module $\psi:\cA\ra S$ has a product \eqref{eq:ab}
resembling what is called a derivation in algebra.
In fact, since $\psi$ is left and right equivariant we have
\[
\psi(a\psi(b))=\psi(a)\psi(b)=\psi(\psi(a)b)\,,
\]
so it makes sense to define a product in $\cA\,$:
\begin{equation}\label{eq:ab}
ab=a\psi(b)+\psi(a)b\,.
\end{equation}
Then because the addition is fiberwise we have $\psi(ab)=\psi(a)\psi(b)\,$.
The product \eqref{eq:ab} is associative:
we have
\[
a(bc)=a\psi(bc)+\psi(a)(bc)
=a\psi(b)\psi(c)+\psi(a)(b\psi(c)+\psi(b)c)
\]
\[
=a\psi(b)\psi(c)+\psi(a)(b\psi(c))+\psi(a)\psi(b)c
\overbrace{=}^{\text{bal.}}a\psi(b)\psi(c)+(\psi(a)b)\psi(c)+\psi(a)\psi(b)c
\]
\[
=(a\psi(b)+\psi(a)b)\psi(c)+\psi(a)\psi(b)c
=ab\psi(c)+\psi(ab)c=(ab)c\,.
\]
The $0$-map $\zero$ is a homomorphism for \eqref{eq:ab} too:
\[
\zero(r)\zero(s)=\zero(r)\psi(\zero(s))+\psi(\zero(r))\zero(s)
=\zero(r)s+r\zero(s)
\]
\[
=\zero(rs)+\zero(rs)
=\zero(rs)\,\text{(Rem.\ \ref{rem:zero})}.
\]
Using the product \eqref{eq:ab} we also have
\begin{equation}\label{eq:a0r}
a\zero(r)
=a\psi(\zero(r))+\psi(a)\zero(r)
=ar + \zero(\psi(a)r)
=ar+\zero(\psi(ar))
=ar\,.
\end{equation}
In \eqref{eq:a0r} we have used that $\zero$ is a section of $\psi\,$,
that $\zero$ is also left equivariant, and the 0-law.
It follows from \eqref{eq:a0r} and associativity of  \eqref{eq:ab} that $a(br)=(ab)r\,$.
Observe that $(ab)^*=b^*a^*$ holds: the left action $ra$ is tautologically defined as $(a^*r^*)^*$
(Rem.\ \ref{rem:bimod}) so we have
\[
(ab)^*=(a\psi(b)+\psi(a)b)^*=(a\psi(b))^*+(\psi(a)b)^*
=\psi(b^*)a^*+b^*\psi(a^*)=b^*a^*\,.
\]
Finally, we have
\[
a(br)^*=a(b\zero(r))^*=a(\zero(r^*)b^*)=(a\zero(r^*))b^*=(ar^*)b^*\,.
\]
\end{remark}

\begin{remark}
Generally, we cannot expect that the product \eqref{eq:ab} makes a balanced $S$-module
into a (balanced) $S$-algebra because for instance distributivity over addition is not expected.
However, Prop.\ \ref{prop:idem} explains a special case when we do get an $S$-algebra.
\end{remark}

\begin{proposition}\label{prop:idem}
Let $\psi:\cA\ra S$ be a balanced $S$-module.
If addition in $\cA$ is idempotent {\em ($a+a=a$)\/},
then with the product \eqref{eq:ab} $\psi$
is a balanced $S$-algebra
such that $aa^*=a\psi(a^*)\,$.
\end{proposition}
\begin{proof}
Distributivity over addition: suppose $\psi(a)=\psi(b)=\psi(a+b)\,$.
Then we have
\[
(a+b)c=(a+b)\psi(c)+\psi(a+b)c=a\psi(c)+b\psi(c)+\psi(a)c
\]
\[
=a\psi(c)+b\psi(c) + (\psi(a)c + \psi(a)c)
=a\psi(c)+b\psi(c) + (\psi(a)c + \psi(b)c)
\]
\[
=a\psi(c)+\psi(a)c+b\psi(c)+\psi(b)c=ac+bc\,.
\]
When $\psi$ is balanced for every $a\in \cA\,$, $a\psi(a^*)$ is self-adjoint:
\[
a\psi(a^*)=(a\psi(a^*))^*=\psi(a)a^*\,.
\]
Therefore,
\[
\psi(a)(a^*\psi(a))=(\psi(a)a^*)\psi(a)=a\psi(a^*)\psi(a)=a\,.
\]
Then we have
\[
aa^*a=a\psi(a^*a)+\psi(a)a^*a
=a\psi(a^*)\psi(a)+\psi(a)(a^*\psi(a)+\psi(a^*)a)\,,
\]
\[
=a\psi(a^*)\psi(a)+\psi(a)(a^*\psi(a))+\psi(a)\psi(a^*)a
=a + a + a=a\,.
\]
Also
\[
aa^*=a\psi(a^*)+\psi(a)a^*=a\psi(a^*)+a\psi(a^*)=a\psi(a^*)\,.
\]
All other requirements are already established.
\end{proof}

\subsection{Free involutive $S$-modules}\label{subsec:free}
\leavevmode\par\vspace{1ex}

Suppose that $X$ is a left involutive semigroup and
\[
f:X\ra S
\]
is an \'etale $*$-homomorphism.
Thus, $f$ is a typical object of the topos $\B(S)\,$.
We shall construct the free balanced $S$-module on $f\,$.
Consider the collection of formal finite fiberwise sums of elements of $X\,$:
\[
F(f)=\{\, x_1+\cdots+x_n\mid f(x_1)=\cdots=f(x_n)\,\}\,.
\]
By definition, such a formal sum lies in the same fiber as its summands.
Thus, there is a map
\[
\widehat{f}:F(f)\ra S\;;\;f(x_1+\cdots+x_n)= f(x_1)=\cdots=f(x_n)\,.
\]
We define $(x_1+\cdots+x_n)^*=x_1^*+\cdots+x_n^*\,$.

We define an action by $S$ in $F(f)$ by means of
the canonical lifting action associated with an \'etale $*$-homomorphism
(Def.\ \ref{def:action} and Lem.\ \ref{lem:xfy}):
let $xs$ denote this action as in \S~\ref{subsec:action}.
Let $x=x_1+\cdots+x_n$ be an element of $F(f)\,$,
where $\forall i\,f(x_i)=r\,$.
Define
\[
xs=(x_1+\cdots+x_n)s=x_1s+\cdots+x_ns\,,
\]
where for every $i\,$, $x_is$ is the unique right lifting of $rs=rr^*rs=f(x_ix_i^*)rs\,$.
Thus, for every $i$ we have $f(x_is)=rs$ and $x_is=x_ix_i^*(x_is)\,$.

\begin{proposition}
$\widehat{f}:F(f)\ra S$ is a balanced $S$-set.
\end{proposition}
\begin{proof}
$\widehat{f}$ is a $*$-morphism, and $\widehat{f}$ is equivariant:
$\widehat{f}(xs)=rs=\widehat{f}(x)s\,$.
Moreover, the conditions of Def.\ \ref{def:bal} lift from $f$ to $\widehat{f}\,$.
\end{proof}

\begin{remark}
By definition, the fiber $\widehat{f}^{-1}(r)$ is the free commutative monoid on $f^{-1}(r)\,$.
\end{remark}

\begin{proposition}
For any left involutive semigroup $X$ and \'etale $*$-homomorphism $f:X\ra S\,$,
$\widehat{f}:F(f)\ra S$ is a balanced $S$-module.
\end{proposition}
\begin{proof}
Let $\zero(r)$ denote the empty formal sum in the fiber of $r\,$.
$F(f)$ has an obvious fiberwise commutative addition.
We see easily that all the conditions of Def.\ \ref{def:invmod} are met.
\end{proof}

We have a commutative diagram
\begin{equation}\label{eq:fhat}
\xymatrix{X \ar[rr]^-\rho \ar[dr]^-f && F(f) \ar[dl]^-{\widehat{f}}\\
& S}
\end{equation}
where $\rho(x)=x\,$.
Then $\rho$ is an equivariant $*$-morphism.

\begin{remark}
We would like to make a balanced $S$-algebra from the free
balanced $S$-module on $f$ using the product \eqref{eq:ab}.
By Prop.\ \ref{prop:idem} the idempotent quotient of $\widehat{f}$
ought to be an involutive $S$-algebra.
We shall see in \S~\ref{subsec:var} that it is.
\end{remark}

\subsection{A quotient of a free involutive $S$-module}\label{subsec:var}
\leavevmode\par\vspace{1ex}

Throughout, $X$ denotes a left involutive semigroup,
and $f:X\ra S$ an \'etale $*$-homomorphism.
We shall construct a quotient of the free balanced $S$-module on $f$
in which the addition is idempotent.
An element of this quotient is a formal sum $x_1+\cdots+x_n$ of elements of $X$ in which
every $x_i$ is distinct:
thus, we may identify $x_1+\cdots+x_n$ simply as a finite subset of $X\,$:
\[
A=\{x_1,\ldots,x_n\}\,.
\]
Addition is union of subsets.
However, in keeping with \S~\ref{subsec:free}
we shall continue to use additive notation:
thus $A+B$ means union.
Also we write a singleton subset $\{x\}$ just as $x$ in order to
conform with the additive notation.
Since the addition is idempotent, by Prop.\ \ref{prop:idem}
the quotient algebra is an involutive $S$-algebra with the product \eqref{eq:ab}.
There is a connecting left $*$-homomorphism $\rho$ (Prop.\ \ref{prop:rho}).

Let
\[
\widehat{F}(f)=\{ (r,A)\mid A\; \text{is finite and}\; A\subseteq f^{-1}(r)\,\}\,.
\]
Of course, the empty set $\emptyset$ is a finite set,
so for any $r\in S\,$, $(r,\emptyset)\in \widehat{F}(f)\,$.
If $A\not=\emptyset\,$, then we shall denote an element $(r,A)$ just by $A$
because $r$ is uniquely determined.
We denote $(r,\emptyset)$ by $\zero(r)\,$.
We define $\widehat{f}(r,A)=r\,$.
In particular, $\widehat{f}(\zero(r))=r\,$.
Thus, $\widehat{f}$ and $\zero$ are two functions as follows,
where $\zero$ is a section of $\widehat{f}\,$.
\begin{equation}\label{eq:hatf}
\xymatrix{S \ar[rr]^-{\zero} \ar@{=}[dr] && {\widehat{F}(f)} \ar[dl]^-{\widehat{f}} \\ & S}
\end{equation}
For $A\in \widehat{F}(f)$ let
$A^*=\{a^*\mid a\in A\,\}\,$.
Then $A^*$ is again an element of $\widehat{F}(f)\,$,
and $\widehat{f}(A^*)=\widehat{f}(A)^*\,$.
For any $A\in \widehat{F}(f)\,$, with $\widehat{f}(A)=r\,$, and any $s\in S$ let
\begin{equation}\label{eq:As}
As=\{ as \mid as\;\text{is the unique right lifting of}\; rs=f(aa^*)rs\,,\;a\in A\,\}\,.
\end{equation}
Any $as\in As$ satisfies $as=aa^*(as)$ in $X\,$.
Then $As$ is finite and
$\widehat{f}(As)=rs=\widehat{f}(A)s\,$.
For any $B\in \widehat{F}(f)\,$, such that $\widehat{f}(B)=s\,$, and any $r$ let
\[
rB=(B^*r^*)^*
\]
\[
=\{ (b^*r^*)^*\mid b^*r^*\;\text{is the unique right lifting of}\; (rs)^*=f(b^*b)(rs)^*\,,\,b\in B\,\}
\]
\[
=\{ rb \mid (rb)^*\;\text{is the unique right lifting of}\; (rs)^*=f(b^*b)(rs)^*\,,\,b\in B\,\}\,.
\]
Any such $b^*r^*$ satisfies $b^*r^*=b^*b(b^*r^*)$ in $X\,$,
and $f(b^*r^*)=(rs)^*\,$, so that $\widehat{f}(rB)=rs=r\widehat{f}(B)\,$.
We have $(rB)^*=B^*r^*\,$, and $(As)^*=s^*A^*\,$.

\begin{exercise}\label{ex:Ar}
For any $A\in \widehat{F}(f)\,$, with $\widehat{f}(A)=r\,$, show that $Ar^*r=A=rr^*A\,$.
Show that
$Ar^*=\{ aa^*\mid a\in A\,\}\,$.
Thus, $Ar^*=(Ar^*)^*=rA^*\,$.
Show that
\[
r^*A=\{ a^*a\mid a\in A\,\}=(r^*A)^*=A^*r\,.
\]
\end{exercise}

Thus, we have an involution and fiberwise addition (which is union) on $\widehat{F}(f)$
with right action \eqref{eq:As} making
\eqref{eq:hatf} a balanced $S$-module in which addition is idempotent.

We denote the product in $\widehat{F}(f)$ given generally by \eqref{eq:ab} as:
\begin{equation}\label{eq:rAsB}
AB=A\widehat{f}(B)+ \widehat{f}(A)B=As \overbrace{+}^{\text{union}} rB\,,
\end{equation}
where $\widehat{f}(A)=r\,$, $\widehat{f}(B)=s\,$,
and $\widehat{f}(AB)=rs\,$.
The addition is idempotent, so by Prop.\ \ref{prop:idem}, with product \eqref{eq:rAsB},
\eqref{eq:hatf} is a balanced $S$-algebra.

\begin{exercise}
Show that $As=A\zero(s)\,$, as predicted by the general theory (Rem.\ \ref{rem:canon}).
\end{exercise}

\begin{example}\label{eg:aa}
For any $A\in \widehat{F}(f)\,$, if $\widehat{f}(A)=r\,$, then we have
\[
A^*A=A^*r + r^*A=A^*r=r^*A\,,
\]
by Ex.\ \ref{ex:Ar}.
Then
\[
AA^*A=A(r^*A)=A\widehat{f}(r^*A)+\widehat{f}(A)r^*A=Ar^*r+rr^*A=A+A=A\,.
\]
Thus, every element of $\widehat{F}(f)$ is a partial isometry,
as predicted in Prop.\ \ref{prop:idem},
so that its underlying multiplicative semigroup
is an involutive semigroup.
\end{example}

\begin{exercise}[Projections of $\widehat{F}(f)$]
Show that $A\in \widehat{F}(f)$ is a projection ($A^*A=A$) if and only if every element of $A$ is a projection of $X\,$.
\end{exercise}

Diagram \eqref{eq:fhat} may be factored as follows,
where $\rho(x)=x\,$.
\begin{equation}\label{eq:fhat2}
\xymatrix{&& F(f) \ar@{>>}[d] \ar[ddl]|\hole \\
X \ar[urr]^-\rho \ar[rr]^-\rho \ar[dr]^-f && {\widehat{F}}(f) \ar[dl]^-{\widehat{f}}\\
& S}
\end{equation}
Just as in the free case $\widehat{f}$ is a $*$-homomorphism,
and $\rho$ is an injective $*$-morphism.
Both maps are equivariant as before.

\begin{proposition}\label{prop:rho}
$\rho:X\ra \widehat{F}(f)$ is a left $*$-homomorphism
with respect to the product \eqref{eq:rAsB} in $\widehat{F}(f)\,$.
\end{proposition}
\begin{proof}
To say that $\rho$ is a left homomorphism is to say that for any $x,y\in X$ we have
\begin{equation}\label{eq:xyxy}
xy=xf(x^*xy)+f(x)x^*xy\,,
\end{equation}
where by definition $f(x)x^*xy$ means $((x^*xy)^*f(x)^*)^*\,$.
By Lem.\ \ref{lem:xfy} we have $xf(x^*xy)=x(x^*xy)=xy\,$.
By Lem.\ \ref{lem:xfy} again we also have
\[
(x^*xy)^*f(x^*)=(x^*xy)^*x^*=(xy)^*\,,
\]
so that $f(x)x^*xy=((x^*xy)^*f(x)^*)^*=xy\,$.
This proves \eqref{eq:xyxy} because $xy+xy=xy\,$.
\end{proof}

A morphism of \'etale left involutive semigroups over $S$ passes
to an involutive $S$-algebra morphism in the usual way.
\begin{equation}\label{eq:x1n}
\xymatrix{
X \ar[d]_-\rho \ar[rr]^-{\varphi} \ar[ddr]|\hole^<<<<<<f && Y \ar[d]^-\rho \ar[ddl]|\hole_<<<<<<{g}\\
{\widehat{F}}(f)\ar[dr]_-{\widehat{f}}  \ar[rr]^-{\widehat{\varphi}} && {\widehat{F}}(g)   \ar[dl]^-{\widehat{g}}\\
& S}
\widehat{\varphi}(x_1+\cdots+x_n)\equiv\varphi(x_1)+\cdots+\varphi(x_n)
\end{equation}
Of course, when thinking of $x_1+\cdots+x_n$ as a subset $\{x_1,\ldots,x_n\}$
then the subset $\widehat{\varphi}(x_1+\cdots+x_n)$ is
given by the image $\{\,\varphi(x_i)\mid 1\leq i\leq n\,\}\,$:
this is what $\equiv$ in \eqref{eq:x1n} means.

We pass an involutive $S$-algebra $\psi:\cA\ra S$ to a presheaf on $L(S)$
by probing it with left $*$-homomorphisms just as we did with semigroups.
\begin{equation}
\xymatrix{
S(e) \ar[dr]  \ar[rr] && \cA  \ar[dl]_-{\psi}\\
& S}
\end{equation}
For such a $\psi$ a presheaf
\[
\Gamma(\psi)(e)=\{\, \text{left $*$-homomorphisms}\; S(e)\ra \cA\;\text{over}\; S\,\}
\]
is defined.
Ultimately, we have a functor $\Gamma$ from involutive $S$-algebras to $\B(S)\,$,
depicted in \eqref{eq:fun}.
\begin{equation}\label{eq:fun}
\xymatrix{& S\text{-Alg} \ar@/^2ex/[dr]^-{\Gamma}  \\
\text{\'Etale}/S  \ar@/^2ex/[ur]^-{\widehat{F}} \ar@/^2ex/[rr]^-\Gamma   && \ar@/^2ex/[ll]^-\Lambda \B(S) }
\end{equation}
The functors $\Lambda\adj \Gamma$ depicted horizontally in \eqref{eq:fun} form an equivalence.

\end{document}